\documentclass[final,reqno,onefignum,onetabnum]{siamltex1213}

\usepackage{subdepth}
\usepackage{overpic}
\usepackage{amsmath,amstext,amsbsy,amssymb,mathdots}
\usepackage{booktabs,bm}
\usepackage{mathrsfs}
\usepackage{tikz,cite}
\usepackage{todonotes}
\usepackage{courier}
\usetikzlibrary{positioning}
\usetikzlibrary{shapes,arrows}
\usepackage[fleqn,tbtags]{mathtools}
\usepackage{amsfonts}
\usepackage{relsize}
\usepackage{xcolor,colortbl}
\usepackage{psfrag}
\usepackage{alltt}
\usepackage{arydshln}
\usepackage{enumerate}
\usepackage{epstopdf}
\usepackage{enumitem}

\newcommand{\even}{\textnormal{{\tiny +}}}
\newcommand{\odd}{\textnormal{{\tiny --}}}
\newcommand{\feven}{f^{\even}}
\newcommand{\fodd}{f^{\odd}}
\newcommand{\fteven}{\tilde{f}^{\even}}
\newcommand{\ftodd}{\tilde{f}^{\odd}}
\newcommand{\eteven}{\tilde{e}^{\even}}
\newcommand{\etodd}{\tilde{e}^{\odd}}
\newcommand{\ceven}{{c}^{\even}}
\newcommand{\codd}{{c}^{\odd}}
\newcommand{\reven}{{r}^{\even}}
\newcommand{\rodd}{{r}^{\odd}}

\newcommand{\meven}{m^{\even}}
\newcommand{\modd}{m^{\odd}}

\newcommand{\Keven}{K^{\even}}
\newcommand{\Kodd}{K^{\odd}}
\newcommand{\deven}{d^{\even}}
\newcommand{\dodd}{d^{\odd}}
\newcommand{\Minv}{M^{\dagger_{\epsilon}}}

\title{Computing with functions in spherical and polar geometries II. The disk}
\author{Heather Wilber\thanks{Center for Applied Mathematics, Cornell University, Ithaca, NY  14853. (hdw27@cornell.edu). This work is supported by a grant from the NASA Idaho Space Grant Consortium.} \and Alex Townsend\thanks{Department of Mathematics, Cornell University, Ithaca, NY  14853. (townsend@cornell.edu). This work is supported by National Science Foundation grant No.~1522577.} \and Grady B. Wright\thanks{Department of Mathematics, Boise State University, Boise, ID 83725-1555.  (gradywright@boisestate.edu). This work is supported by National Science Foundation grant DMS 1160379.}}

\begin{document}
\maketitle

\begin{abstract}
A collection of algorithms is described for numerically computing with smooth functions defined on the unit disk. Low rank approximations 
to functions in polar geometries are formed by synthesizing the disk analogue of the double Fourier sphere method with a structure-preserving 
variant of iterative Gaussian elimination that is shown to converge geometrically for certain analytic functions. This adaptive procedure is near-optimal in its sampling strategy, producing approximants that are stable for differentiation and facilitate the use of FFT-based algorithms in both variables. The low rank form of the approximants is especially useful for operations such as integration and differentiation, reducing them to essentially 1D procedures,  and it is also exploited to formulate a new fast disk Poisson solver that computes low rank approximations to solutions. This work complements a companion paper (Part I) on computing with functions on the surface of the unit sphere. 
\end{abstract}



\section{Introduction} 
Polar geometries play a central role in scientific computing, with applications in fluid dynamics~\cite{kerswell2005recent, serre2001three}, 
optics~\cite{mahajan2007orthonormal}, and astrophysics~\cite{pringle1981accretion, godon1997numerical}. Advances in these areas  require  effective representations for functions on the unit disk, and compressed representations of such functions have become increasingly important.  
We develop a novel variant of iterative Gaussian elimination (GE) that adaptively constructs low rank approximants with near-optimal compression properties; this enables fast and spectrally accurate computations with functions on the disk. 

Methods that represent functions on the disk with expansions in the Chebyshev--Fourier basis allow  for the use of fast transforms~\cite{eisen1991spectral,  Fornberg_95_01, shen2000new}, but may not maintain regularity at the origin of the disk when used with GE. 
Alternatively, representations employing expansions that incorporate regularity in the basis are not readily associated with fast transforms~\cite{vasil2016tensor}. 
Unsatisfied with having to choose between either regularity at the origin or fast transforms, we propose an approach that attempts to prioritize both. Combining low rank function approximation with an interpolation method that 
samples functions over the unit disk in a way that is analogous to the double Fourier sphere (DFS) method~\cite{Fornberg_95_01}, we construct 
approximants with several desirable properties: (1) A structure that permits the use of fast transforms based on the fast Fourier transform (FFT) 
in both variables, (2) regularity over the origin of the disk, and (3) a near-optimal underlying interpolation grid that does not oversample near 
the origin. 

Using this idea, we have created an integrated computational framework for working with functions in polar geometries. This includes the 
development of algorithms for integration, function evaluation, vector calculus, and a fast Poisson solver. Our software is 
publicly available through the open source Chebfun software system written in MATLAB~\cite{Chebfun}.
This development allows investigators to compute in polar geometries without concern for the underlying discretization or procedural details, providing 
an intuitive platform for data-driven computations, explorations and visualizations with functions on the unit disk. Various examples are 
available at www.chebfun.org/examples for the reader to explore.

Part I of this two-part series of papers developed a structure-preserving, iterative variant of Gaussian elimination (GE) for computing with 
functions on the surface of the unit sphere~\cite{townsend2015computing}. Here, we extend the ideas of~\cite{townsend2015computing} 
to functions defined on the unit disk.  We also include several new results that were not discussed in Part I. In Section~\ref{sec: convergence}, we prove that our structure-preserving GE procedure converges geometrically for functions that are analytic in a sufficiently large region in the complex plane.  Section~\ref{sec:Poisson} describes a new Poisson solver that constructs near-optimal low rank approximations to solutions, and is conceptually quite different from the Poisson solver described in~\cite{townsend2015computing}. Additional new results include  a weighted singular value decomposition algorithm (Section~\ref{sec:svd}), and an extended discussion on the near-optimality of the GE procedure (Section~\ref{sec: BMCsvd}). 

The paper is structured as follows: First, we review existing techniques for computing with functions on the disk (Section~\ref{sec:existing}), 
including a discussion of the disk analogue to the DFS method. A brief review of low rank function approximation in Section~\ref{sec:low rank approx} is 
followed by a detailed description of the structure-preserving GE procedure applied to functions on the disk.
A collection of fast algorithms for computing with the resulting low rank approximants is given in Section~\ref{sec:diskalgorithms}, 
and a fast disk Poisson solver for computing solutions in low rank form is described in Section~\ref{sec:Poisson}. 

\section{Existing techniques for computations on the disk} \label{sec:existing}
There is an extensive literature on numerical methods for computing with functions on the disk. An overview in the context 
of solving Poisson's equation is given in~\cite{boyd2011comparing}. We briefly review a selection of these strategies. 

\subsection{Radial basis functions}
As a mesh-free method, radial basis functions can be used for applications on many types of geometries~\cite{FFBook}. Specific studies of 
global approximations on the disk include~\cite{Heryudono2010,Karageorghis2007304}, where 
the interpolation points are arranged so that the computational cost of the method reduces 
from $\mathcal{O}(N^3)$ to $\mathcal{O}(N\log N)$ operations, where $N$ is the number of function samples taken. Ill-conditioning can cause a loss of 3-5 digits of accuracy in problems of moderate size, but in most applications, this is perfectly acceptable. However, this prevents the construction of approximants that are accurate to machine precision, which is what we require.  

\subsection{Conformal mapping} Using the inverse of the cosine leminiscate function, a function $f$ on the unit disk can be mapped conformally 
to the unit square~\cite{schwarz1869ueber, amore2008solving}. This mapping avoids introducing a potentially problematic singularity at the origin and allows $f$ 
to be expressed as a bivariate Chebyshev expansion so that FFT-based transforms are applicable. 
Unfortunately, the mapping introduces four new artificial singularities corresponding to the corners of the square. Interpolation points 
unnaturally cluster near these singularities, resulting in excessive oversampling that diminishes the computational efficiency gained 
from the use of the FFT.\; In contrast, our approach enables the use of FFT-based transforms, while employing low rank approximation to avoid 
overresolving functions near the origin. 
     
\subsection{Basis expansions} A function $f(x, y)$ defined in Cartesian coordinates on the unit disk 
can be converted to a function in polar coordinates, $f(\theta, \rho )$, through the transformation
\begin{equation}
\label{eq:transfvars}
x = \rho\cos\theta,\quad y = \rho\sin\theta, \qquad (\theta, \rho ) \in [-\pi, \pi] \times [0, 1].
\end{equation}
This change of variables relates a function on the disk to a function defined on a rectangular domain, 
where advantageous algorithms can often be employed. Noting that functions on the disk are periodic in the angular variable, $\theta$, a sufficiently smooth function $f$ can be approximated by a Fourier expansion: 

\begin{equation}
\label{eq: Fourierexpansion}
f(\theta, \rho) \approx \sum_{k=-n/2}^{n/2-1} \phi_k(\rho) e^{ik\theta}, \qquad (\theta,\rho)\in[-\pi,\pi]\times[0,1],
\end{equation}
where $n$ is an even integer. It is not obvious what expansion should be employed 
for representing the function $\phi_k(\rho)$. Three common choices are:

\begin{itemize}[leftmargin=*]

\item {\bf Bessel expansions:} A natural analogue of the trigonometric and spherical harmonic expansions, Bessel
expansions are derived from the eigenfunctions of the Laplace operator in polar coordinates~\cite{churchill1961fourier}. 
Here, assuming that $f(\theta,1) = 0$ for $\theta\in[-\pi,\pi]$, we write $\phi_k(\rho) = \sum_{\ell =0}^{m-1} a_{\ell k} J_k(\omega_{k \ell} \rho)$,  $\rho \in [0, 1]$, 
where $J_k(z)$ is the $k$th order Bessel function, and $\omega_{k \ell}$ is the $\ell$th positive root of $J_k(z)$~\cite[(10.23)]{NIST}. 
The expansion can also be modified to allow for functions that are nonzero at the boundary of the disk.
This choice guarantees the expansion is smooth at the origin, but to compute the expansion coefficients, one must approximate integrals involving Bessel functions. While fast algorithms for such computations exist, they are particularly effective only when the parameter $k$ is small~\cite{kapurl1995algorithm, townsend2015fast}.  More generalized algorithms typically involve significant precomputational costs~\cite{o2010algorithm}, and this limits their effectiveness in a regime where functions are resolved on adaptive grids. 

 \item {\bf One-sided Jacobi polynomial expansions:} Writing $\phi_k(\rho)$ as an expansion over the one-sided Jacobi 
 polynomials results in an expansion of $f(\theta, \rho)$ in the Zernike polynomial 
 basis~\cite{bhatia1954circle, von1934beugungstheorie}. This set of polynomials 
is considered theoretically analogous to the Legendre polynomials due to its orthogonality properties~\cite{bhatia1954circle}, 
and is often the basis of choice for approximation on the disk. 
More recently, a whole hierarchy of bases related to the one-sided Jacobi polynomials were employed to capture the 
regularity of vector- and tensor-valued functions on the disk~\cite{vasil2016tensor}. As before, this choice guarantees the 
expansion is smooth at the origin, but fast algorithms for computing the expansion coefficients are not efficient in our setting due to precomputational costs~\cite{o2010algorithm}. 

 \item {\bf Chebyshev expansions:} 
 Expanding $\phi_k(\rho)$ in the Chebyshev basis results in a truncated Chebyshev--Fourier expansion of $f$, i.e., 
\begin{equation}
\label{eq:FourierChebyshev}
f(\theta,\rho) \approx \sum_{k=-n/2}^{n/2-1} \sum_{\ell=0}^{m-1} a_{\ell k}T_\ell(2\rho-1)e^{ik\theta}, \qquad (\theta, \rho) \in [-\pi, \pi] \times [0, 1],
\end{equation}
where $T_\ell$ is the degree $\ell$ Chebyshev polynomial defined on $[-1,1]$. Given samples of $f$ on an $m \times n$ Chebyshev--Fourier tensor product grid over $[-\pi, \pi] \times [0, 1]$, 
the coefficients in~\eqref{eq:FourierChebyshev} 
can be computed in $\mathcal{O}(mn\log(mn))$ operations via the FFT.\; Unfortunately, this grid is artificially clustered near 
$\rho=0$~\cite{Fornberg_95_01}, and this choice of basis does not naturally impose any regularity at $\rho=0$. Our 
approach alleviates both of these drawbacks by combining the disk analogue to the DFS (see Section~\ref{sec:DFSdisk}) with a structure-preserving 
low rank construction procedure (see Section~\ref{sec:low rank approx}).

\end{itemize}

\subsection{The disk analogue of the double Fourier sphere method}\label{sec:DFSdisk}
 The disk analogue of the DFS method proceeds by constructing a Chebyshev--Fourier expansion of a function defined on $[-\pi, \pi] \times [-1, 1]$, 
 instead of $[-\pi, \pi] \times [0, 1]$. This strategy ``doubles" $f$ over the disk in the sense that $f$ is sampled 
 twice, but  $\rho=0$ is no longer treated as a boundary. Mathematically, this doubled extension of $f$, which we will call $\tilde{f}$, can be expressed by defining $g(\theta, \rho)$ and $h(\theta, \rho)$ on $[0, \pi] \times [0, 1]$, so 
that $g(\theta, \rho) = f(\theta-\pi, \rho)$ and $h(\theta, \rho) = f(\theta, \rho)$.  Then, 
 \begin{equation}
 \tilde{f}(\theta, \rho) =
\begin{cases}
g(\theta + \pi, \rho), & (\theta, \rho) \in [-\pi, 0] \times [0, 1],\\
h(\theta, \rho), & (\theta, \rho) \in [0, \pi] \times [0, 1],\\
g(\theta, -\rho), & (\theta, \rho) \in [0, \pi] \times [-1, 0],\\
h(\theta + \pi, -\rho), & (\theta, \rho) \in [-\pi, 0] \times [-1, 0]. 
\end{cases}
\label{eq:BMCsym disk}
\end{equation}
This idea is conceptually analogous to the DFS method~\cite{Merilees_73_01}, which is used for approximating functions on the surface of the unit sphere~\cite{townsend2015computing}.  

A useful connection between the DFS method and its disk analogue is the presence of similar structure in the extended functions. 
We observe in~\eqref{eq:BMCsym disk} that $ \tilde{f}$ possesses \textit{block-mirror centrosymmetric} (BMC) structure~\cite{townsend2015computing}, and refer to functions that satisfy~\eqref{eq:BMCsym disk} as BMC functions. 

The BMC structure of $ \tilde{f}$ 
can be intuitively described as
\begin{equation}
 \tilde{f} = \begin{bmatrix} g & h \\[3pt] {\tt flip}(h) & {\tt flip}(g) \end{bmatrix},  
\label{eq:depictBMC1}
\end{equation} 
where ${\tt flip}$ refers to the MATLAB command that reverses the 
order of the rows of a matrix. This is also called a glide reflection in group theory~\cite[\S 8.1]{martin2012transformation}.
 
\begin{figure} 
 \centering
 \begin{minipage}{.49\textwidth} 
 \centering
  \begin{overpic}[width=.55\textwidth]{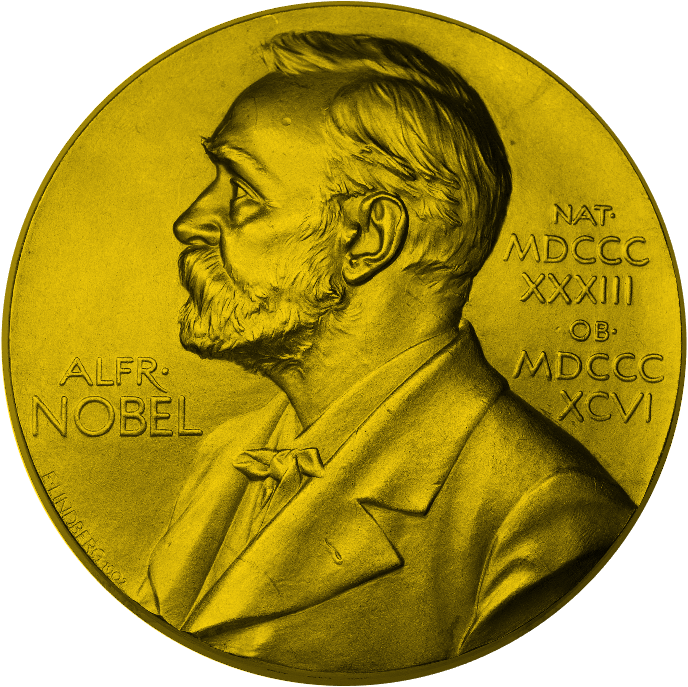}
  \end{overpic}
  
  \begin{overpic}[width=\textwidth]{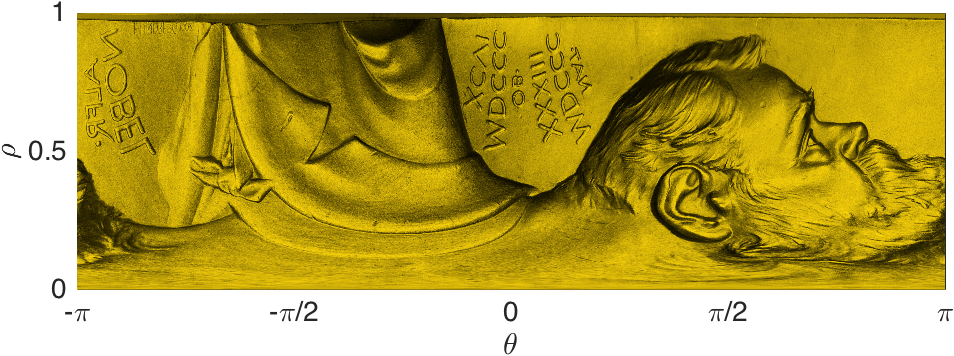}
  \put(15,95) {(a)}
  \put(5,40) {(b)}
  \end{overpic}
 \end{minipage}
\begin{minipage}{.49\textwidth} 
  \begin{overpic}[width=\textwidth]{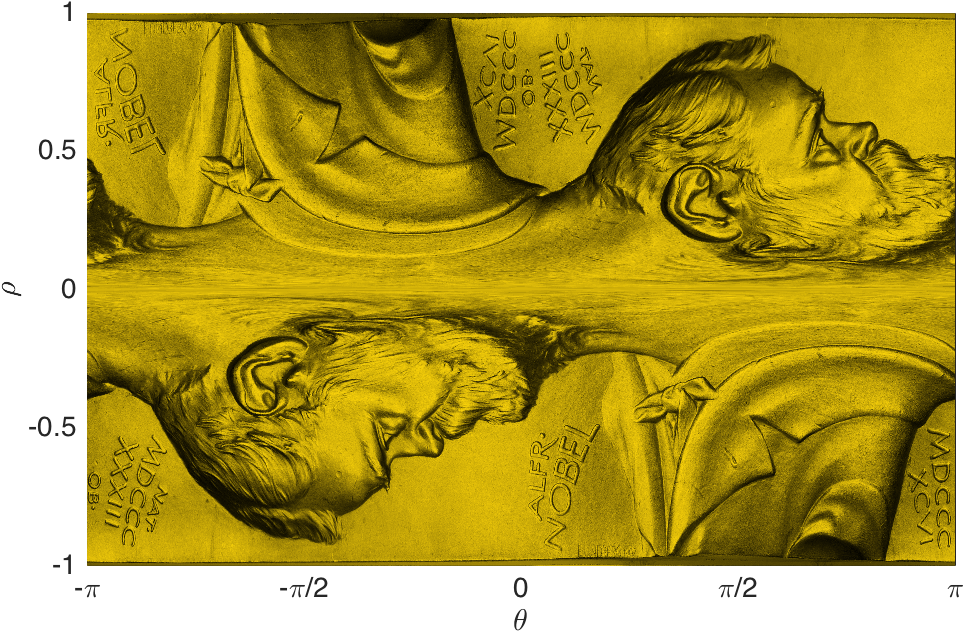}
 \put(15,70) {(c)}
 \end{overpic}
 \end{minipage}
 \caption{The disk analogue of the DFS method applied to the Nobel prize medal. (a) The medal. (b) The projection of the medal using polar coordinates. (c) The medal after applying the disk analogue to the DFS method. This is a BMC-II ``function'' that is periodic in $\theta$ and defined over $\rho\in[-1,1]$.}
 \label{fig:Coin}
\end{figure}

In addition to having BMC structure and being periodic in $\theta$, $ \tilde{f}$ must be constant along the line representing the origin of the 
disk, $\rho =0$. This feature of $ \tilde{f}$ is not shared by all BMC functions. For example, the BMC function 
$\tilde{f}(\theta, \rho) = \sin2\theta \cos2\rho$ is not constant along $\tilde{f}(\theta, 0)$ for $\theta \in [-\pi, \pi]$, and therefore does not correspond 
to a continuous function on the disk. To capture this important aspect of BMC functions associated with the disk, we define the following variant: 
\begin{definition}(BMC-II function)
A function $\tilde{f}:[-\pi,\pi]\times[-1,1]\rightarrow \mathbb{C}$ is a Type-II BMC (BMC-II) function if it is a BMC function and 
$f(\cdot,0) = \alpha$, where $\alpha$ is a constant.  
\label{def:BMCfunctionII}
\end{definition}

An analogous variant for computing on the sphere, the BMC-I function, 
is defined to be constant along two lines corresponding to the north and south poles of the sphere~\cite{townsend2015computing}.

Figure~\ref{fig:Coin} displays the analogue of the DFS method applied to the Nobel Prize medal and illustrates BMC-II structure. Since every function $f$ on the 
disk corresponds to a BMC-II function $\tilde{f}$ that is $2\pi$-periodic in $\theta$, we apply our approximation technique and all subsequent algorithms on 
$\tilde{f}$, with rigid adherence to preserving the BMC-II structure at every step. Calculations performed on $\tilde{f}$ always 
correspond to a computation on the original function, $f$, and consistently remain associated with the geometry of the disk. For example, smooth functions with BMC-II structure are always continuously differentiable over $\rho = 0$. In Section~\ref{sec:diskalgorithms}, we discuss the differentiation of BMC-II functions in more detail. 

The strategy of doubling up interpolation grids on the disk to reduce the redundancy of sampling near $\rho=0$ in spectral collocation 
methods is well established~\cite{Fornberg_95_01,trefethen2000spectral}, and several variants have been 
proposed~\cite{heinrichs2004spectral,eisen1991spectral,shen2000new}.
These doubling strategies alleviate some, but not all, of the issues associated with oversampling near the origin. Our approach is different 
in that it combines a doubling strategy with a low rank approximation procedure. Low rank methods provide compressed representations of 
functions and can therefore further alleviate issues related to the  overresolution of functions near the origin of the disk (see Figure~\ref{fig:diskgrid}). 
 
 \subsection{Software} 
 \label{sec:software} 
 Our software for computing with functions on the unit disk is called Diskfun.\footnote{After our software was developed and posted on GitHub, another software system named ``diskfun" was 
 released in the Approxfun software system written in Julia. It is not related to this work.}\;It is implemented within MATLAB as a part of Chebfun~\cite{Chebfun}, and is accessed through the creation 
 of objects called diskfuns. Below, we display the MATLAB code used to 
 represent the function 
\[f(\theta, \rho) = \cos\big(3\pi \rho \big) + \sin\big(2\rho\sin\theta-.4\big)\]  as a diskfun object: 
 \begin{verbatim}
 	 f = diskfun(@(t,r) cos(3*pi*r)+sin(2*r.*sin(t)-.4),'polar')
 	 f =
     diskfun object:
       domain        rank    vertical scale
      unit disk       13           2
\end{verbatim}
The printout provides the numerical \textit{rank} of the function, discussed in Section~\ref{sec:low rank approx}, and it also displays the vertical 
scale, an approximation of the absolute maximum value of $f$. 
 
The default setting of Diskfun assumes that functions are supplied in Cartesian coordinates. However, diskfun objects can be constructed from function 
handles in polar
coordinates by adding the flag \texttt{`polar'} to the construction command, as shown above. Once a diskfun is created, users have access to a large number of 
algorithms tailored to functions defined on the disk via overloaded MATLAB commands (see Section~\ref{sec:diskalgorithms}). 
For example, integration of $f$ is performed by the \texttt{sum} command, and differentiation is performed by \texttt{diff}. 
 
 \section{Low rank approximation for functions on the disk} \label{sec:low rank approx}
In~\cite{Chebfun2}, a low rank approximation method for computing with 2D functions on bounded rectangular domains is described. The authors
construct compressed representations of bivariate functions that facilitate the use of essentially 1D algorithms in subsequent computations. This 
makes it especially useful in relation to Chebfun, where efficient 1D procedures are well established and highly optimized. 
Here, we develop an analogous technique for the polar setting. 

A nonzero function $ \tilde{f}(\theta,\rho)$ is a rank 1 function if it can be written as a product of two univariate functions, i.e., 
$\tilde{f}(\theta,\rho) = c(\rho)r(\theta)$. A function $\tilde{f}$ is of rank at most $K$ if it can be written as a sum of $K$ rank 1 functions. 
While most functions are mathematically of infinite rank, smooth functions can often be approximated to machine precision with a rank $K$ 
truncation, i.e.,
\begin{equation}
\label{eq:genlowrankapprox}
\tilde{f}(\theta, \rho) \approx \sum_{j=1}^{K} c_j(\rho)r_j(\theta),
\end{equation}
for some relatively small $K$~\cite{Chebfun2}. Below, we develop an efficient procedure for constructing rank $K$ approximants of 
BMC-II functions that preserves BMC-II structure.

\subsection{Iterative Gaussian elimination on functions}

Given a matrix $A$ of rank $n$, $K<n$ steps of Gaussian elimination (GE) with complete or rook pivoting can often be used to construct a near-best rank~$K$ approximation to $A$, provided that the singular values of $A$ decay to zero sufficiently fast~\cite{foster2006comparison}. Methods related to GE, such as pseudoskeleton approximation~\cite{goreinov1997theory}, adaptive cross approximation~\cite{bebendorf2000approximation}, two-sided interpolative 
decomposition~\cite{halko2011finding}, and Geddes--Newton approximation~\cite{carvajal2005hybrid} can be used to find low rank approximations to multivariate functions. In~\cite{Chebfun2}, such approximations are constructed using an adaptive, iterative variant of GE with complete pivoting, and we will extend this idea to the approximation of functions in polar geometries.

Given the function $\tilde{f}$, denote the maximum absolute value of $\tilde{f}$ for $(\theta, \rho) \in [-\pi, \pi] \times [-1, 1]$ by 
$\tilde{f}(\theta^*, \rho^*)$. This value serves as a pivot. A GE step with complete pivoting proceeds by forming a rank 1 function from 
this pivot and subtracting it from $\tilde{f}$: 
\begin{equation} 
\tilde{f}(\theta,\rho) \quad \longleftarrow\quad \tilde{f}(\theta,\rho) - \underbrace{\frac{\tilde{f}(\theta^*,\rho)\tilde{f}(\theta,\rho^*)}{\tilde{f}(\theta^*,\rho^*)}}_{\text{A rank~$1$ approx.~to $\tilde{f}$}}.
\label{eq:GEstep}
\end{equation}      
In this scheme, functions of the form $\tilde{f}(\theta^*,\rho)$ are referred to as ``column slices" of $\tilde{f}$. Similarly, functions of 
the form $\tilde{f}(\theta,\rho^*)$ are ``row slices". The step in~\eqref{eq:GEstep} zeros out the row and column slices containing the pivot. 
Since $\tilde{f}$ may be of infinite rank, the GE procedure is terminated after the absolute maximum of the residual falls below some specified 
relative tolerance, such as the product of machine epsilon and the (approximate) maximum value of the function. The number of steps required to achieve this is an upper bound on the \textit{numerical rank} of $\tilde{f}$, 
which is the minimum rank required to approximate $\tilde{f}$ to machine precision using any bounded function of finite rank~\cite{townsend2014computing}.

Applying the GE procedure to $ \tilde{f}$ for $K$ steps, a rank $K$ approximation is constructed:
\begin{equation}
 \tilde{f} (\theta, \rho) \approx \sum_{j=1}^{K}d_j c_j(\rho) r_j(\theta).
\label{eq: low_rank_rep} 
\end{equation}
Here, $d_j$ is a coefficient related to the GE pivots, and $c_j(\rho)$ and $r_j(\theta)$ are the $j$th column slice and row slice, 
respectively, constructed during the GE procedure. 

Unfortunately, this GE procedure does not preserve BMC-II symmetry and therefore destroys the association between $\tilde{f}$ and a continuous 
function on the disk. In~\cite{townsend2015computing}, a variation of GE that preserves symmetry is described for BMC functions 
related to the sphere. Crucially, this method only depends on the BMC structure of the function, and not on any additional features related 
to spherical geometries per se. With some modifications, as we now describe, this procedure also applies to BMC-II functions associated with the disk. 

\subsection{Structure-preserving Gaussian elimination} \label{sec:structGE}
The structure-preserving GE algorithm presented in~\cite{townsend2015computing} performs a GE step similar to~\eqref{eq:GEstep}, but with 
the scalar pivot replaced with the following $2\times2$ pivot matrix:
\begin{figure} 
\centering
\includegraphics[width=.36\textwidth,trim=220 400 220 210, clip]{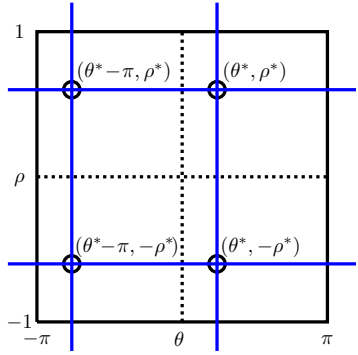}
\caption{A $2\times 2$ pivot (black circles) and corresponding column and row slices (blue lines) used in a GE step on $\tilde{f}$ to 
preserve the BMC structure of a function.}
\label{fig:pivotLocations}
\end{figure}
\begin{equation}
\label{eq:pivotblockcoords}
M = \begin{bmatrix}
 \tilde{f}(\theta^*-\pi, \rho^*) &  \tilde{f}(\theta^*, \rho^*) \\  \tilde{f}(\theta^*-\pi, -\rho^*) &  \tilde{f}(\theta^*, -\rho^*)
\end{bmatrix},
\end{equation}
where $(\theta^*, \rho^*) \in [0, \pi] \times [0, 1]$ are fixed values selected by the pivoting strategy described in 
Figure~\ref{fig:GaussianElimination}. To understand why this is an appropriate choice, note that BMC symmetry is entirely characterized 
by the following two equalities: $\tilde{f}(\theta^*-\pi, \rho)= \tilde{f}(\theta^*, -\rho)$, $\rho \in [-1, 1]$, and 
$\tilde{f}(\theta, \rho^*) = \tilde{f}(\theta-\pi, -\rho^*)$, $\theta \in [-\pi, \pi]$. Figure~\ref{fig:pivotLocations} shows that the location 
of the entries of $M$ correspond to the intersections of these row and column slices. Letting $\tilde{f}(\theta^*-\pi, \rho^*)=a$ and 
$\tilde{f}(\theta^*, \rho^*) = b$, (\ref{eq:pivotblockcoords}) can be written as the centrosymmetric matrix 
\begin{equation}
\label{eq:pivotblockvals}
M = \begin{bmatrix}
a & b \\ b & a
\end{bmatrix}.
\end{equation} 
Assuming $M^{-1}$ exists, a GE step with the pivot matrix $M$ is given by 
\begin{equation}
  \tilde{f}(\theta, \rho) \quad \longleftarrow\quad  \tilde{f}(\theta, \rho) - \underbrace{\begin{bmatrix} \tilde{f}(\theta^*-\pi,\rho) &  \tilde{f}(\theta^*,\rho)\\[3pt] \end{bmatrix} M^{-1} \begin{bmatrix} \tilde{f}(\theta,\rho^*)\\[3pt]  \tilde{f}(\theta,-\rho^*)\end{bmatrix}}_{\displaystyle =\tilde{s}(\theta, \rho)}.
\label{eq:GEstructurePreservingStep}
\end{equation}
We now show that the GE step in~\eqref{eq:GEstructurePreservingStep} preserves BMC symmetry of $\tilde{f}$.

\begin{lemma} \label{lem:GEstepsym} 
Given a BMC function $\tilde{f}$, the update $\tilde{s}$ in~\eqref{eq:GEstructurePreservingStep} is also a BMC 
function. That is, the GE step in~\eqref{eq:GEstructurePreservingStep} preserves BMC-symmetry. 
\end{lemma}
\begin{proof} 
To show that $\tilde{s}(\theta, \rho)$ has BMC structure, we employ {\em quasimatrices}.\footnote{A quasimatrix $A$ of size $[a, b] \times n$ is a matrix with $n$ columns, where each 
column is a function defined on the interval $[a, b]$~\cite{Townsend_15_01}.} 

Let $J$ denote the $2 \times 2$ exchange matrix, so that for a matrix $A \in \mathbb{C}^{2 \times n}$, $JA$ reverses the rows of $A$. 
Let $\mathcal{J}$ be the reflection operator, $\mathcal{J}: \tilde{s}(\theta, \rho) \rightarrow \tilde{s}(\theta, -\rho)$. 
Now we use blocks of quasimatrices to rewrite $\tilde{s}$. Writing $\tilde{f}$ in terms of the functions $g$ and $h$ given 
in~\eqref{eq:BMCsym disk}, we have 
$M= \begin{bmatrix}
  g(\theta^{*}, \rho^{*}) &  h(\theta^{*}, \rho^{*})     \\ 
 h(\theta^{*}, \rho^{*}) &  g(\theta^{*}, \rho^{*})    
\end{bmatrix}$.
Let $Q$ be the $[0, \pi] \times 2$ quasimatrix defined as $Q = \begin{bmatrix} g(\theta^*, \rho) & | & h(\theta^*, \rho)\end{bmatrix}$, 
and let $P$ be the $[0, 1] \times 2$ quasimatrix defined as $P = \begin{bmatrix} \text g(\theta, \rho^{*}) & |& h(\theta, \rho^{*})\end{bmatrix}$. 
Then, $\tilde{s}$ in~\eqref{eq:GEstructurePreservingStep} can be written as 
\begin{equation}
\label{eq:exchblocks}
\tilde{s} = 
\begin{bmatrix}
Q \\
\mathcal{J} (QJ) \\	
\end{bmatrix}
M^{-1}
\begin{bmatrix}
P^T & JP^T  
\end{bmatrix}. 
\end{equation}
Since $M^{-1}$ is centrosymmetric, it commutes with $J$. Using this fact,~\eqref{eq:exchblocks} becomes
\begin{equation}
\tilde{s} = 
 \begin{bmatrix} 
 QM^{-1}P^T & QM^{-1}JP^T \\
 \mathcal{J} (QM^{-1}JP^T ) & \mathcal{J} ( QM^{-1}P^T )
 \end{bmatrix}
 \label{eq:blockBMC},
\end{equation}
which, by the definition of $\mathcal{J}$, is a BMC function. 
\end{proof}

Lemma~\ref{lem:GEstepsym} demonstrates that~\eqref{eq:GEstructurePreservingStep} provides a structure-preserving GE procedure for BMC 
functions that can be used to construct a low rank approximation to $\tilde{f}$ as in~\eqref{eq: low_rank_rep}.\footnote{The function $\tilde{s}$ in~\eqref{eq:GEstructurePreservingStep} is rank $2$ and can be split into two rank $1$ BMC functions (see Section~\ref{sec:parity}).}  However, this relies on the 
fact that $M$ is invertible, which may not always be the case. For example, $M$ is singular for any BMC function that is $\pi$--periodic 
in $\theta$. For this reason, we must replace $M^{-1}$ in~\eqref{eq:GEstructurePreservingStep} with $\Minv$, the $\epsilon$-pseudoinverse 
of $M$~\cite[Sec.~5.5.2]{Golub_2012_01}. The matrix $\Minv$ is associated with the singular values of $M$ and a parameter $\epsilon>0$. 
We will discuss the choice of $\epsilon$ in Section~\ref{sec:parity}, and an explicit formula for $\Minv$ is given in~\cite{townsend2015computing}. 
Using $\Minv$, the amended GE step is expressed by
\begin{equation}
 \tilde{f}(\theta,\rho) \quad \longleftarrow\quad \tilde{f}(\theta,\rho) - \begin{bmatrix}\tilde{f}(\theta^*-\pi,\rho) & \tilde{f}(\theta^*,\rho)\\[3pt] \end{bmatrix} \Minv \begin{bmatrix}\tilde{f}(\theta,\rho^*)\\[3pt] \tilde{f}(\theta,-\rho^*)\end{bmatrix}.
\label{eq:GEstructurePreservingStepPseudoinverse}
\end{equation} 
Lemma~\ref{lem:GEstepsym} also holds for~\eqref{eq:GEstructurePreservingStepPseudoinverse} because, like $M^{-1}$, $\Minv$ is centrosymmetric. 

The strategy used to select each pivot matrix is important, as it relates to the efficiency and convergence of the GE procedure. The $2\times2$ 
analogue of complete pivoting proceeds by choosing $(\theta^*, \rho^*) \in [0, \pi] \times [0, 1]$ such that $\sigma_1(M)$ is maximized over all 
$M$, where $\sigma_1(M)$ is the larger of the two singular values of $M$. Given the simple form of $M$ in~\eqref{eq:pivotblockvals}, it is easy 
to see that $ \sigma_1(M) =\max\{|a+b|, |a-b|\}$. In practice, it is much more efficient to choose $(\theta^*, \rho^*)$ 
from a coarse, discrete grid on $[-\pi,\pi ] \times [0, 1]$. This results in a large, but not necessarily maximal, value of $\sigma_1(M)$. 
Fortunately, GE is robust to these kinds of compromises, as a detailed analysis in~\cite{townsend2016gaussian} shows. 

\begin{figure} 
 \centering 
 \fbox{
 \parbox{.97\textwidth}{ 
 \textbf{Algorithm: Structure-preserving GE on BMC functions}
 
 \vspace{.2cm}
 
 \textbf{Input:} A BMC function $\tilde{f}$ and a coupling parameter $0\leq \alpha\leq 1$.\\[-7pt] 
 
 \textbf{Output:} A structure-preserving low rank approximation $\tilde{f}_k$ to $\tilde{f}$.\\[-3pt]
 
 Set $\tilde{f}_0=0$ and $\tilde{e}_0 = \tilde{f}$.\\[-7pt]
  
 \textbf{for} $k=1,2,3,\ldots,$\\[-10pt]
 
 $\quad$ Find $(\theta_{k},\rho_{k})$ such that $M =  \begin{bmatrix}a & b \cr b & a\end{bmatrix}$, where $a = \tilde{e}_{k-1}(\theta_{k-1}-\pi,\rho_{k-1})$ and 
 
 $\quad$ $b = \tilde{e}_{k-1}(\theta_{k-1},\rho_{k-1})$ has maximal $\sigma_1(M)$.\\[-7pt]
  
 $\quad$ Set $\epsilon = \alpha\sigma_1(M)$. \\[-10pt]
   
 $\quad$ $\tilde{e}_{k} = \tilde{e}_{k-1} - \begin{bmatrix}\tilde{e}_{k-1}(\theta_k-\pi,\rho) & \tilde{e}_{k-1}(\theta_k,\rho)\\[3pt] \end{bmatrix} \Minv \begin{bmatrix}\tilde{e}_{k-1}(\theta,\rho_k)\\[3pt] \tilde{e}_{k-1}(\theta,-\rho_k)\end{bmatrix}$.\\
 
 $\quad$ $\tilde{f}_{k} = \tilde{f}_{k-1} - \begin{bmatrix}\tilde{e}_{k-1}(\theta_k-\pi,\rho) & \tilde{e}_{k-1}(\theta_k,\rho)\\[3pt] \end{bmatrix} \Minv \begin{bmatrix}\tilde{e}_{k-1}(\theta,\rho_k)\\[3pt] \tilde{e}_{k-1}(\theta,-\rho_k)\end{bmatrix}$.\\[-7pt]
  
 \textbf{end}
 }}
 \caption{A continuous idealization of our structure-preserving GE procedure on BMC 
 functions. In practice we use a discretization of this procedure and terminate it after a finite 
 number of steps.}
\label{fig:GaussianElimination}
\end{figure}

The above GE procedure preserves \textit{general} BMC structure, but it does not preserve BMC-II structure: Nothing 
in~\eqref{eq:GEstructurePreservingStepPseudoinverse} enforces that each constructed rank $1$ function in~\eqref{eq: low_rank_rep} 
is constant along the line $\tilde{f}(\theta, 0)$. However, in the case where $\tilde{f}(\theta, 0)=0$, each term 
in~\eqref{eq: low_rank_rep} constructed through~\eqref{eq:GEstructurePreservingStepPseudoinverse} will possess BMC-II structure. This suggests 
a strategy for the case where $\tilde{f}(\theta, 0) \not= 0$. Since $\tilde{f}(\theta, 0)$ is constant by Definition~\ref{def:BMCfunctionII},
we deliberately choose the first GE step to zero out $\tilde{f}(\theta, 0)$ by subtracting off a rank 1 term that is constant in the $\theta$ direction: 
\begin{equation}
\tilde{f}(\theta, \rho) \quad \longleftarrow \quad \tilde{f}(\theta, \rho) - \tilde{f}(\theta^*, \rho). 
\label{eq: GE zero pole step}
\end{equation}
Since the update to $\tilde{f}$ is zero along $\tilde{f}(\theta, 0)$ after this modification, each additional rank~1 term constructed through 
continued applications of~\eqref{eq:GEstructurePreservingStepPseudoinverse} possesses BMC-II structure. 

A continuous idealization of the BMC-preserving GE process is shown in Figure~\ref{fig:GaussianElimination}. In practice, the algorithm implemented 
in Diskfun proceeds in two phases; this process is identical to the method described in~\cite{Chebfun2}, except with $2\times2$ pivots. The result is 
a low rank approximation to $\tilde{f}$ of the form~\eqref{eq: low_rank_rep}. We represent each of the $r_j(\theta)$ and $c_j(\rho)$ 
functions in~\eqref{eq: low_rank_rep} using Fourier and Chebyshev interpolants, respectively.
This process is achieved in \smash{$\mathcal{O}(K^3+K^2(m+n))$} operations~\cite{Chebfun2}, where $K$ is the numerical rank of the function, 
and $m$ and $n$ are the maximum number of Chebyshev and Fourier coefficients required to resolve the functions $c_j(\rho)$ and $r_j(\theta)$, respectively, 
to machine precision. 
\begin{figure} 
\begin{center}
 \begin{minipage}{.4\textwidth}
  \includegraphics[width=0.8\textwidth]{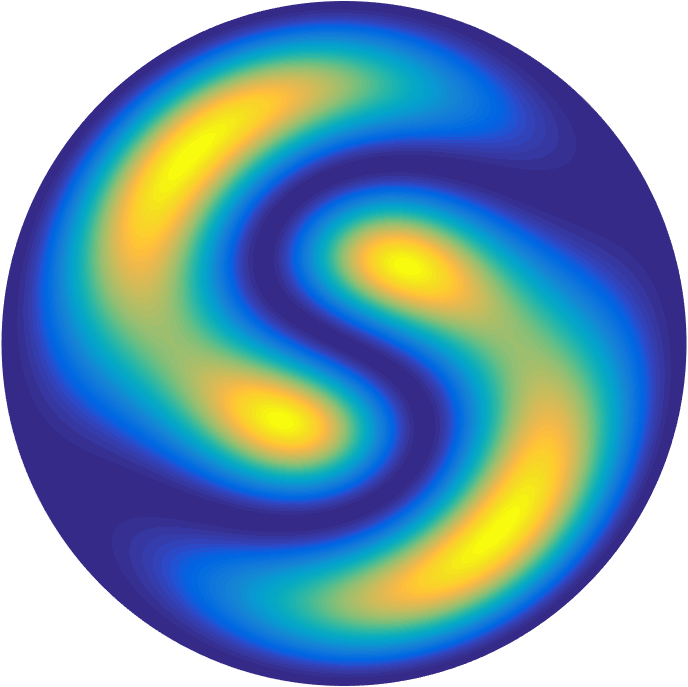}
 \end{minipage}
 \begin{minipage}{.4\textwidth}
  \includegraphics[width=0.8\textwidth]{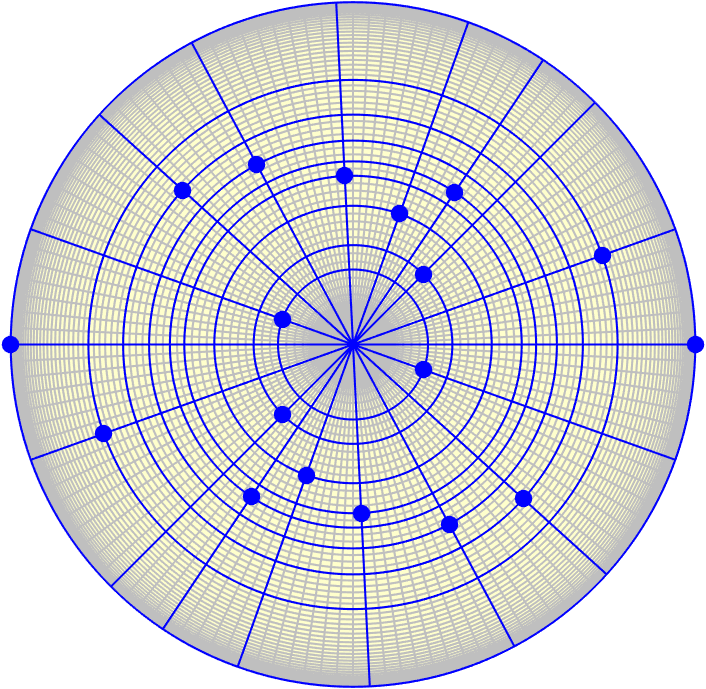}
 \end{minipage}
 \end{center}
 \caption{Left: The function $f(\theta,\rho)=-\cos((\sin(\pi \rho) \cos(\theta) + \sin(2\pi\rho)\sin(\theta))/4)$ on the unit disk, constructed 
with the \texttt{diskfun} command \texttt{f = diskfun(@(t,r) -cos((sin(pi*r).*cos(t)+sin(2*pi*r).*sin(t))/4),'polar')}
and plotted with the command \texttt{plot(f)}. Right: The skeleton used to approximate $f$, plotted with the command \texttt{plot(f,'.-')}.
The blue dots are the pivot locations taken by GE.~~The GE procedure samples $f$ at $m$ Chebyshev points along each blue line, and $n$ equispaced points along each blue circle,  where $m$ and $n$ correspond to number of Chebyshev coefficients and Fourier modes, respectively, in~\eqref{eq: low_rank_rep}.  The underlying tensor product grid (in gray) shows the sample points required to approximate $f$ to machine precision without the GE procedure applied to the DFS method. The overresolution of the tensor grid over the low rank skeleton can be seen.}
\label{fig:diskgrid} 
\end{figure}

The example in Figure~\ref{fig:diskgrid} illustrates the form of the final approximant.  
Each $c_j(\rho)$ defines a radial ``slice" of the function, 
and each $r_j(\theta)$ defines a circular ``slice". 
To form these slices, the GE algorithm adaptively samples $\tilde{f}$ along a sparse collection of lines referred to as the \textit{skeleton}, and constructs a rank $K$ approximant of the form of~\eqref{eq: low_rank_rep}.  In this process, only $K^2 + K(m+n)$ samples are required to approximate $\tilde{f}$ to machine precision, as opposed to the $mn$ samples required for the tensor product. 
As depicted in Figure~\ref{fig:diskgrid}, the use of low rank methods effectively counters the overresolution issues associated with applying Chebyshev--Fourier tensor product grids on the disk.

\subsection{A parity-based interpretation of structure-preserving GE} \label{sec:parity}
For an approximation to a function $f$ on the disk to be continuous and differentiable 
at $\rho=0$, the following properties must hold for the Fourier expansion of $f$ given in~\eqref{eq: Fourierexpansion}: 
\begin{itemize}
\item[(i)] $k$ is even $\implies \phi_k(\rho)$ is an even function,
\item[(ii)] $k$ is odd $\implies \phi_k(\rho)$ is an odd function,
\item[(iii)] $k \not=0$ $\implies \phi_k(0)=0$. 
\end{itemize}
In this section we show that these parity properties can be naturally recovered for the BMC-II function $\tilde{f}$, and are preserved by 
the GE procedure depicted in Figure~\ref{fig:GaussianElimination}. 

Let $\tilde{f}$ be a BMC function defined via functions $g$ and $h$ in~\eqref{eq:BMCsym disk}. Let $\feven = g + h$ and $\fodd = g-h$. 
Then, $\tilde{f}$ can be written as a sum of two BMC functions~\cite[Section 3.2]{townsend2015computing}:
\begin{align}
\tilde{f} = 
\frac12
\underbrace{
\begin{bmatrix}
\feven & \feven \\
{\tt flip}(\feven) & {\tt flip}(\feven)
\end{bmatrix}}_{\displaystyle =\fteven}
+
\frac12
\underbrace{
\begin{bmatrix}
\fodd & -\fodd \\
-{\tt flip}(\fodd) & {\tt flip}(\fodd)
\end{bmatrix}}_{\displaystyle =\ftodd},
\label{eq:EvenOddSplit}
\end{align}
i.e., $\tilde{f} = \frac12(\fteven + \ftodd)$. From~\eqref{eq:EvenOddSplit}, we can deduce that $\fteven$ is an even function in $\rho$ and 
$\pi$-periodic in $\theta$, whereas $\ftodd$ is an odd function in $\rho$ and $\pi$-antiperiodic in $\theta$. This is equivalent to the 
statement of parity properties (i) and (ii), as $\pi$-periodic functions have only even Fourier modes and $\pi$-antiperiodic functions 
have only odd Fourier modes.
While many techniques enforce these parity-based restrictions on the Fourier and Chebyshev coefficients of functions on the disk, 
relating these properties more generally to BMC-II functions allows one to apply these restrictions directly through the \emph{values} of 
a function, without ever using the coefficients. This is the premise our GE procedure operates on. 

As shown in Section 3.2 of~\cite{townsend2015computing}, we can write the GE step~\eqref{eq:GEstructurePreservingStepPseudoinverse} as
\begin{equation}
  \tilde{f}(\theta,\rho) \; \longleftarrow\;  \frac12(\fteven(\theta,\rho) - \meven\fteven(\theta^*,\rho)\fteven(\theta,\rho^*)) + \frac12(\ftodd(\theta,\rho) - \modd\ftodd(\theta^*,\rho)\ftodd(\theta,\rho^*)),
\label{eq:FirstGEstepSplit}
\end{equation}
where $\meven$ and $\modd$ are values\footnote{ Note that $\meven$ and $\modd$ are not related to  $m$ in~\eqref{eq:FourierChebyshev}.} derived from the spectral decomposition of $\Minv$, and are given by
\begin{align}
(\meven,\modd) =
\begin{cases}
(1/(a+b),0), & \text{if $|a-b|<\alpha |a+b|$}, \\
(0,1/(a-b)), & \text{if $|a+b|<\alpha |a-b|$}, \\
(1/(a+b),1/(a-b)), & \text{otherwise}. \\
\end{cases}
\label{eq:mvalues}
\end{align}
Here, $0 < \alpha < 1$ is referred as the \textit{coupling parameter} for the GE procedure, and $\alpha$ determines $\epsilon$ in $\Minv$: 
$\alpha = \epsilon / \sigma_1(M) = \epsilon / \max\{ |a+b|, |a-b|\}$. The decomposition in~\eqref{eq:FirstGEstepSplit} reveals an alternative 
interpretation of structure-preserving GE as a coupled process involving two standard GE procedures. If either of the first two cases 
of~\eqref{eq:mvalues} is chosen, GE with complete pivoting is performed on only one term in~\eqref{eq:FirstGEstepSplit}, resulting in a rank 1 update. In the third 
case of~\eqref{eq:mvalues}, $\Minv = M^{-1}$, and a rank $2$ update is achieved. It is desirable to perform as many rank $2$ updates as 
possible, as this reduces the overall number of pivot searches required by the GE procedure. 
Too small a value of $\alpha$ may allow the use of $M^{-1}$ when it is ill--conditioned, but choosing $\alpha$ too close to $1$ hampers 
the efficiency of the procedure. We have experimented with several values for $\alpha$ and find that $\alpha = 1/100$ works well in practice. The role of $\alpha$ in the convergence rate of the GE procedure is discussed further in Section~\ref{sec: convergence}.

Following~\cite{townsend2015computing}, we can exploit~\eqref{eq:FirstGEstepSplit} to write the low rank approximation to $\tilde{f}$ as 
\begin{align}
 \tilde{f}(\theta, \rho) \approx \sum_{j=1}^K d_j c_j(\rho)r_j(\theta) =  \sum_{j=1}^{\Keven} \deven_j \ceven_j(\rho)\reven_j(\theta) + \sum_{j=1}^{\Kodd} \dodd_j \codd_j(\rho)\rodd_j(\theta),
\label{eq:lowrankRepresentationSplit}
\end{align} 
where $\Keven + \Kodd=K$. Here, the functions $\ceven_j(\rho)$ and $\reven_j(\theta)$ for $1\leq j\leq \Keven$ are even and $\pi$-periodic, 
respectively, while $\codd_j(\rho)$ and $\rodd_j(\theta)$ for $1\leq j\leq \Kodd$ are odd and $\pi$-antiperiodic, respectively. The pivots, 
$\deven$ and $\dodd$, are related to the $2\times2$ pivot matrix given in~\eqref{eq:pivotblockcoords}~\cite{townsend2015computing}. 
If $f$ is non-zero at the origin, the first step of the GE procedure is given by~\eqref{eq: GE zero pole step}. 
This chooses $\ceven_1(\rho) = \tilde{f}(\theta^{*},\rho)$, $\reven_1(\theta)=1$, and $\deven_1=1$, so that for $j > 1$, $c_j(0)=0$. 
Crucially, this ensures that parity property (iii) is preserved in the decomposition.  

Using~\eqref{eq:lowrankRepresentationSplit}, the parity properties of $\tilde{f}$ are given explicitly, and this can be used to simplify 
algorithmic procedures. An example is given in Section~\ref{sec:diskint} on integration. This expression also clarifies why our 
approximants are stable for differentiation (see Section~\ref{sec:diskdiff}). 

\subsection{Convergence}
\label{sec: convergence}

In~\cite{townsend2015computing}, it is shown that  BMC structure-preserving GE exactly 
recovers BMC functions of finite rank. In this section, we prove that for certain analytic functions of infinite rank,  structure-preserving GE  converges at a geometric rate. 
Specifically, we will consider a function $\tilde{f}$ that 
is analytically continuable in at least one variable to a sufficiently large region of the complex plane.
We characterize this region formally using the concept of a \textit{stadium}.
\begin{definition}[Stadium] The stadium $S_\beta$ with radius $\beta > 0$ is the region in the complex plane consisting of all numbers lying at a distance $\leq \beta$ from an interval $[c, d]$, i.e.,
\[
S_\beta = \left\{z\in\mathbb{C} : \inf_{x\in[c,d]} |x - z| \leq \beta \right\}.
\]
\end{definition} To understand convergence, we will view structure-preserving GE as a coupled procedure involving the functions $\tilde{\feven}$ and $\tilde{\fodd}$ defined in Section~\ref{sec:parity}. The proof requires an examination of the error produced after applying the GE step~\eqref{eq:FirstGEstepSplit}, and we see in~\eqref{eq:mvalues} that there are three cases to consider. Bounds on the error are intimately tied to the \textit{growth factors} of the GE procedures that are applied to \smash{$\tilde{\feven}$} and\smash{ $\tilde{\fodd}$}. The growth factors quantify the worst possible increase in the absolute maximum of the function after a rank one update. Geometric convergence can be proven if the size of the  stadium in which $\tilde{f}$ is analytic is large enough to counteract the potential growth induced by GE.

The connection between the region of analyticity and the GE growth factor is made clear in the proof of Theorem~$8.1$ in~\cite{Townsend_15_01}, which shows that iterative GE with complete pivoting as in~\eqref{eq:GEstep} converges geometrically for functions that are analytic within a sufficiently large stadium.
In the first or second case of~\eqref{eq:mvalues}, standard GE with complete pivoting is applied to either \smash{$\fteven$} or \smash{$\ftodd$}, and we may use \smash{Theorem~$8.1$} directly. In the third case of~\eqref{eq:mvalues}, two GE procedures are performed: a GE step with complete pivoting is applied to whichever of the two functions \smash{$\fteven$} or \smash{$\ftodd$} has a larger absolute maximum value, and a GE step with a nonstandard pivoting strategy is applied to the other function. 
If a bound on the growth factor of this nonstandard GE step is known, then as long as \smash{$\tilde{f}$} is assumed to be analytic in an appropriately-sized region of the complex plane, we can apply a mild generalization of Theorem 8.1.  For this reason, we require the following lemma associated with the third case of~\eqref{eq:mvalues}: 

\begin{lemma}
\label{lemma:secondaryGE}
The growth factor for the nonstandard GE procedure applied within BMC structure-preserving GE is bounded above by  $1+\alpha^{-1}$, where $\alpha$ is the coupling parameter  in~\eqref{eq:mvalues}.
\end{lemma}
\begin{proof}
Consider performing one step of BMC structure-preserving GE on $\tilde{f}$ by operating on $\fteven$ and $\ftodd$ from~\eqref{eq:EvenOddSplit}, 
and suppose we are in the third case of~\eqref{eq:mvalues}. Without loss of generality, suppose that  $|\modd| >  |\meven|$. 
Then, $|\meven|= 1/\sigma_1(M)$,  $|\modd| = 1/\sigma_2(M),$ and a nonstandard GE step is performed on $\ftodd$ using the pivot $\modd $. Here, $\sigma_k(M)$ denotes the $k$th singular value of $M$.

After the nonstandard GE step is applied, the supremum norm of the residual  is
\begin{equation}
\label{eq:onestep}
 \| \etodd_1\|_\infty = \left\| \ftodd  - \textnormal{sgn}(\modd)\dfrac{\ftodd (\theta_*, \cdot)
\ftodd (\cdot, \rho_*)}{\sigma_2(M)}  \right\|_\infty \leq \| \ftodd\|_{\infty}  + \dfrac{ \| \ftodd\|_\infty^2 }{\sigma_2(M)},  
\end{equation}
where $(\theta_*, \rho_*) \in [-\pi, \pi] \times [0, 1]$ is the location of the pivot in the first quadrant. 

Since we are in the third case of~\eqref{eq:mvalues}, we have $\sigma_2(M) \geq \alpha\sigma_1(M)$, and therefore $\sigma_2(M) \geq \alpha \|\ftodd \|_\infty$. Applying these results to~\eqref{eq:onestep} gives that
\begin{equation}
 \| \etodd_{1}\|_\infty  \leq (1+\alpha^{-1}) \| \ftodd\|_{\infty} , 
\end{equation}
i.e., the growth factor for  the nonstandard GE step  cannot exceed $1+\alpha^{-1}$. 
\end{proof}

Since bounds on the growth factors are known for each GE procedure applied on \smash{$\fteven$} and \smash{$\ftodd$},  geometric convergence of  the BMC structure-preserving GE can now be proven. A theorem analogous to the one below holds with the roles of $\theta$ and $\rho$ exchanged. 
\begin{theorem}
\label{thm: thetheorem}
Let $\tilde{f}:[-\pi,\pi]\times[-1,1]\rightarrow \mathbb{R}$ be a 
 BMC function such that $\tilde{f}(\theta,\cdot)$ is continuous for any $\theta \in [-\pi,\pi]$ 
 and $\tilde{f}(\cdot,\rho)$ is analytic and uniformly bounded 
 in a stadium $S_\beta$ of radius $\beta=\max(2,1+\alpha^{-1}) 2\pi \kappa$, $\kappa>1$, for any $\rho \in[-1,1]$.
 Then, there exists a constant $C>0$ such that  
\[ \|\tilde{f}-\tilde{f}_k\|_\infty = \|\tilde{e}_k\|_\infty \leq C\mu^{-k},\]
where $\mu = \min\{ \kappa, \alpha^{-1}\}$,  $\alpha$ is the coupling parameter described in~\eqref{eq:mvalues}, and $\tilde{f}_k$ is the approximant constructed after $k$ steps of the BMC structure-preserving GE procedure. 
\end{theorem}

\begin{proof}
For $k \geq 0$,  $\tilde{e}_k$ is a BMC function and can be written as the sum of 
an even $\pi$-periodic and odd $\pi$-antiperiodic function, i.e., $\tilde{e}_k = \eteven_k + \etodd_k$ (see  Section~\ref{sec:parity}). 
Let $\mu = \min\{ \kappa, \alpha^{-1} \}$, 
and choose a constant $C > 0$ so that $ \|\tilde{e}_0^+\|_\infty \leq C/2$ and $ \|\tilde{e}_0^-\|_\infty \leq C/2$. We will show by induction that  $\|e_k\|_\infty \leq C \mu^{-k}$ for all $k>0$.   

 When $k = 0$, 
 $\max\{\|\eteven_0\|_\infty, \|\etodd_0\|_\infty\} \leq C/2$.
 Suppose that for $k > 0$, the following induction hypothesis holds: 
 \begin{equation} 
\label{eq:indhyp} 
 \max\{\|\eteven_k\|_\infty, \|\etodd_k\|_\infty\} \leq (C/2) \mu^{-k}.
 \end{equation}
Consider the next structure-preserving GE step. Using~\eqref{eq:mvalues}, there are three cases to consider.

\underline{Case 1}: Here, $\|\etodd_k\|_\infty < \alpha \|\eteven_k\|_\infty$, and only $\eteven_k$ is updated (see Section~\ref{sec:parity}).
This step is equivalent to performing a standard GE step with complete pivoting as in~\eqref{eq:GEstep} on $\eteven_k$. 
By Theorem~$8.1$ in~\cite{Townsend_15_01}, we have

  \[\|\eteven_{k+1}\|_\infty \leq \kappa^{-1}\|\eteven_k\|_\infty.\] Since $\etodd_{k+1}=\etodd_k$, we find that \[\|\etodd_{k+1}\|_\infty = \|\etodd_k\|_\infty < \alpha \|\eteven_k\|_\infty,\] and using the definition of $\mu$ and~\eqref{eq:indhyp}, we conclude that
\begin{equation}
\label{eq:case1res}
\max\{ \|\eteven_{k+1}\|_\infty, \|\etodd_{k+1}\|_\infty\} \leq \mu^{-1} \max\{\|\eteven_{k}\|_\infty, \|\etodd_{k}\|_\infty\} \leq (C/2)\mu^{-(k+1)}.
\end{equation}

\underline{Case 2}: Here, $\|\eteven_k\|_\infty < \alpha \|\etodd_k\|_\infty$, and only $\etodd_k$ is updated. This is equivalent to Case 1 with the roles of $\eteven_k$ and $\etodd_k$ interchanged. 

\underline{Case 3}: 
Without loss of generality, suppose that $|\modd| > |\meven|$. Then, a standard GE step with complete pivoting  is applied to $\eteven_k$, and a GE step with nonstandard pivoting is performed on $\etodd_k$. 
As in Case 1, we find that    
\[ \|\eteven_{k+1}\|_\infty \leq \kappa^{-1}\|\eteven_k\|_\infty.\]
For $\etodd_{k+1}$ we use the bound on the growth factor from Lemma~\ref{lemma:secondaryGE} to apply a slight generalization of Theorem 8.1 in~\cite{Townsend_15_01}, finding that 
\[ \|\etodd_{k+1}\|_\infty \leq \kappa^{-1}\|\etodd_k\|_\infty.\] 
It follows from the definition of $\mu$ and~\eqref{eq:indhyp} that 
\[ \max\{\|\eteven_{k+1}\|_\infty, \|\etodd_{k+1}\|_\infty \} \leq (C/2)\mu^{-(k+1)}.\]
 By induction, we have that
 \[ \max\{\|\eteven_k\|_\infty, \|\etodd_k\|_\infty\} \leq (C/2) \mu^{-k}, \qquad k \geq 0, \]
and the result follows from the fact that $\|\tilde{e}_k\|_\infty \leq \|\eteven_k\|_\infty+ \|\etodd_k\|_\infty$. 
\end{proof}

The assumptions required on \smash{$\tilde{f}$} in Theorem~\ref{thm: thetheorem} are rather restrictive, as the proof of convergence requires us to consider GE growth rates that account for the worst-case scenario. Empirically, we observe convergence for a much broader class of functions, and at rates that are asymptotically optimal. This is described in the next section. 

\subsection{Near-optimality} \label{sec: BMCsvd}
While Section~\ref{sec: convergence} proves that convergence of the GE procedure in Figure~\ref{fig:GaussianElimination} is geometric when $f$ is analytic in a sufficiently large region of the complex plane, we observe in practice that the procedure converges at near-optimal rates for functions that are only a few times 
differentiable. 

If $\tilde{f}$ is Lipschitz continuous with respect to both variables for $(\theta, \rho) \in [-\pi, \pi]\times [-1, 1]$,
then the best rank $K$ approximation to $\tilde{f}$ is given by the Karhunen-Lo\`{e}ve expansion, also called the
  the singular value decomposition (SVD), of $\tilde{f}$:
     \begin{equation} 
    \label{eq:BMCSVD}
      \tilde{f}(\theta,\rho) = \sum_{j=1}^\infty \sigma_j u_j(\rho)v_j(\theta), \qquad (\theta,\rho)\in [-\pi, \pi]\times [-1, 1]. 
     \end{equation} 
The non-increasing sequence $\sigma_1\geq \sigma_2\geq \cdots$ of real, nonnegative numbers are the \textit{singular values} 
of $\tilde{f}$. The continuous \textit{singular functions} $\{u_j(\rho)\}$ and $\{v_j(\theta)\}$ each form an orthonormal 
set of functions with respect to the standard $L_2$ inner product. A best rank $K$ approximation to $\tilde{f}$, in the sense of the $L_2$ norm, is constructed by 
truncating~\eqref{eq:BMCSVD} after $K$ terms~\cite{Schmidt_1908_01}. 

For reasons closely related to those discussed in Section~\ref{sec:parity}, the SVD preserves the BMC structure of 
$\tilde{f}$~\cite{wilber2016}. Unfortunately, the high cost of computing the SVD makes this an untenable approach for constructing 
low rank approximants to $\tilde{f}$ in practice. 
\begin{figure} 

\begin{minipage}{.49\textwidth}
\begin{overpic}[width=\textwidth]{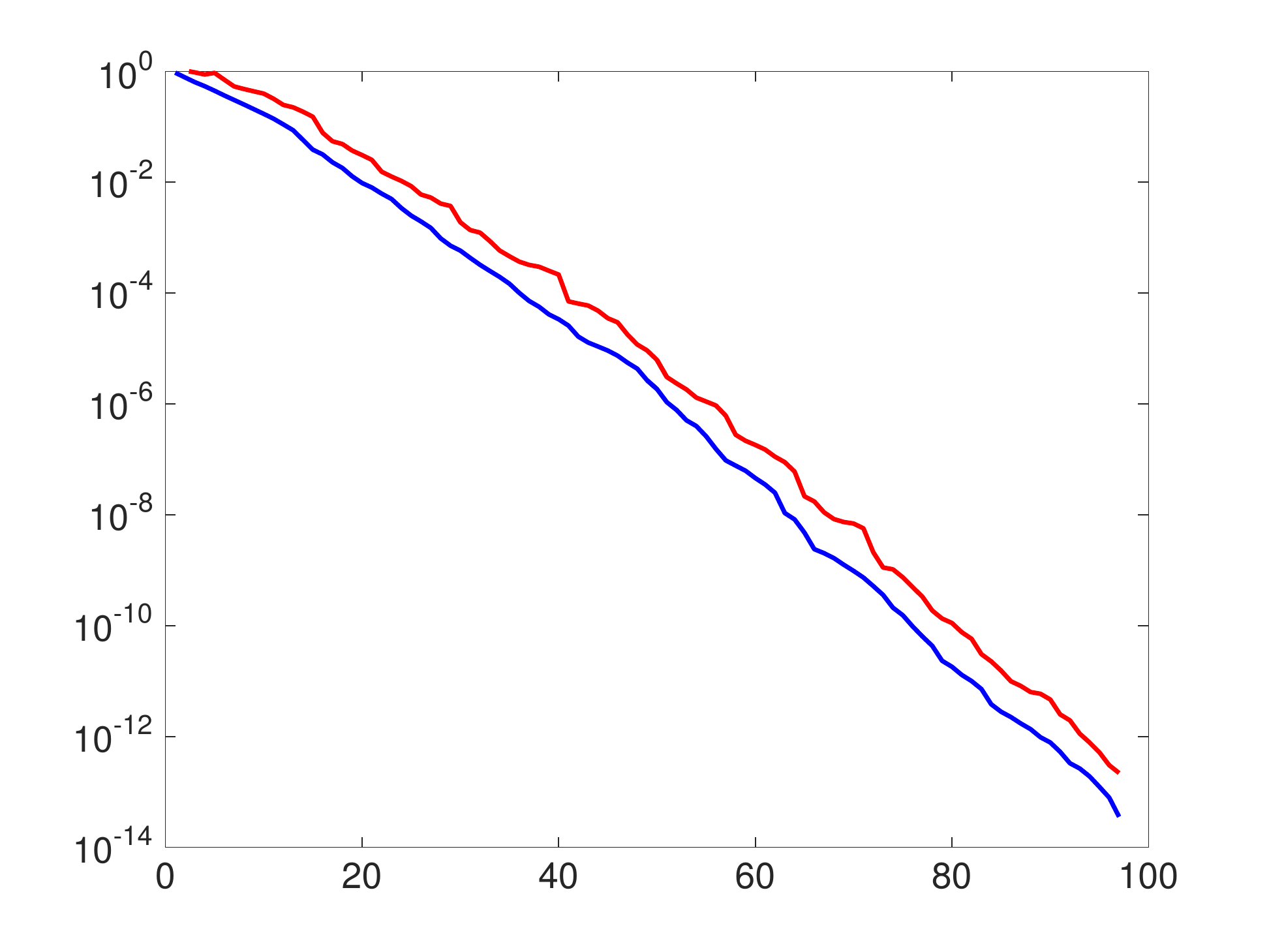} 
 \put(50,72) {$\phi_1$}
 \put(50,50) {\rotatebox{-40}{GE}}
 \put(45,43) {\rotatebox{-40}{SVD}}
 \put(27,0) {Rank of approximant}
 \put(0,30) {\rotatebox{90}{$L_2$ error}}
 \end{overpic}
 \end{minipage}
 \begin{minipage}{.49\textwidth}
 \begin{overpic}[width=\textwidth]{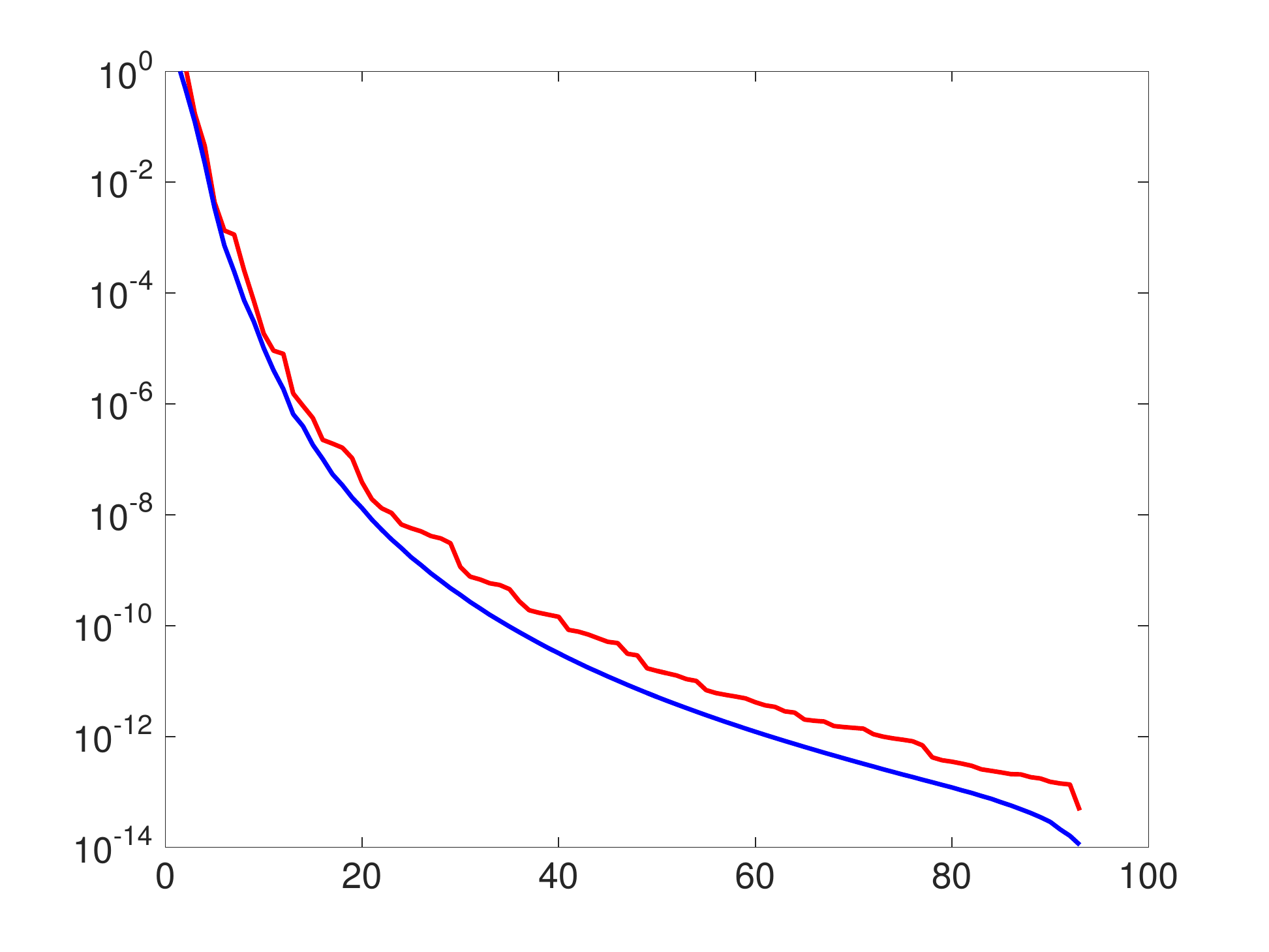} 
  \put(50,72) {$\phi_2$}
 \put(40,30) {\rotatebox{-20}{GE}}
 \put(40,20) {\rotatebox{-22}{SVD}}
 \put(27,0) {Rank of approximant}
 \put(0,30) {\rotatebox{90}{$L_2$ error}}
  \end{overpic}
  \end{minipage}
\caption{A comparison of low rank approximations to the functions in~\eqref{eq:PhiFunctions} computed using the SVD and the iterative GE procedure.  The $L_2$ error is plotted against the rank of the approximants to $\phi_1$ and $\phi_2$.  The $L_2$ error given by the SVD approximants are optimal and we observe that that the low rank approximants constructed by 
the GE procedure are near-optimal.}
\label{fig:SVDvsGE}
\end{figure}
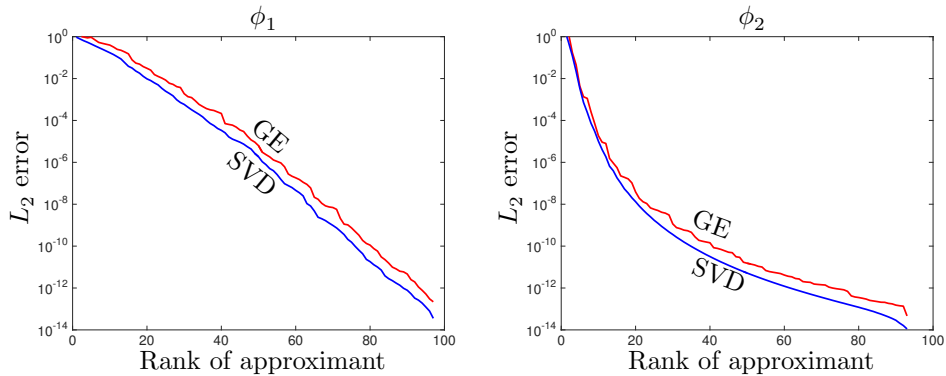
Nonetheless, approximants constructed via the SVD are optimal with respect to $\|\cdot\|_2$, and this provides a way to check the quality of the low rank approximants constructed 
by our GE procedure. Figure~\ref{fig:SVDvsGE} displays the $L_2$ error over $[-\pi,\pi]\times[-1,1]$ for rank $K$ approximations constructed via the SVD and the GE 
procedure for the following two BMC-II functions: 
\begin{equation}
\begin{aligned}
 &\phi_1(\theta,\rho) = \textnormal{exp}\left[-(\cos (11 \rho\sin\theta) + \sin (\rho\cos\theta) )^2\right],\\ 
 &\phi_2(\theta,\rho) = (1-\omega)_{\textnormal{{\tiny +}}}^6\left(35(\omega)^2+18\omega+3\right),
\end{aligned}
\label{eq:PhiFunctions}
\end{equation}
where $\omega(\theta,\rho) = \left((\rho\cos\theta-.2)^2+(\rho\sin\theta-.2)^2\right)^{1/2}$ and $\zeta_{\textnormal{{\tiny +}}} = \max\{\zeta, 0 \}$. 
The error given by the SVD behaves in accordance with known theoretical results, decaying geometrically for the function $\phi_1$ and at an 
algebraic rate for $\phi_2$~\cite{townsend2014computing}. In experiments, it is observed that our GE procedure constructs near-best low rank 
approximants to smooth BMC functions. 

\section{Algorithms for numerical computation with functions on the disk}\label{sec:diskalgorithms} 
In this section, we describe several of the algorithms used in the Diskfun software. These methods rely on the 
fact that every smooth function $f$ on the disk is associated with a BMC-II function $\tilde{f}$ that is periodic in $\theta$. We compute with 
a low rank approximation to $\tilde{f}$ as in~\eqref{eq: low_rank_rep}, which is constructed by the GE procedure in 
Figure~\ref{fig:GaussianElimination}. We rely on the fact that in~\eqref{eq: low_rank_rep}, each $c_j(\rho)$ and $r_j(\theta)$ can be 
approximated by a Chebyshev and Fourier series, respectively, so that for $1 \leq j \leq K$, 
\begin{equation}
c_j(\rho) \approx \sum_{\ell=0}^{m-1} a_{\ell}^j ~T_\ell(\rho), \qquad
r_j(\theta) \approx \sum_{k=-n/2}^{n/2 -1} b_k^j~ e^{i k \theta }, 
\label{eq:1D_expansions}
\end{equation}
where $T_\ell(\rho)$ is the Chebyshev polynomial of degree $\ell$, and $n$ is an even integer.

The algorithms for computing with functions represented in  Chebyshev and Fourier bases differ considerably from one another. 
However, implementation in the Chebfun environment is significantly simplified due to its underlying object-oriented class structure. 
For example, Chebfun overloads commands such as \texttt{sum(g)} (integration) or \texttt{diff(g)} (differentiation), so that the same syntax 
executes different underlying algorithms based on whether the object \texttt{g} is represented by a Chebyshev series or a Fourier 
series~\cite{wright2015extension}. 

\subsection{Pointwise evaluation}\label{sec:diskeval} To efficiently evaluate $ \tilde{f}$ at a fixed point $(\theta_*, \rho_*)$, we 
use~\eqref{eq: low_rank_rep}, observing that
\begin{equation}
 \tilde{f}(\theta_*, \rho_*) \approx \sum_{j=1}^{K} d_j c_j(\rho_*) r_j(\theta_*).  
\label{eq:disk_pointeval} 
\end{equation}
Evaluation of $ \tilde{f}$ proceeds as $2K$ 1D function evaluations. Functions $c_j(\rho )$, $1 \leq j \leq K$, are evaluated using 
Clenshaw's algorithm~\cite[Ch.~19]{trefethen2013approximation}, and functions $r_j(\theta)$, $1 \leq j \leq K$, are evaluated using 
Horner's scheme~\cite{wright2015extension}. Altogether, this requires $\mathcal{O}(K( m+n ))$ operations. The algorithm 
is implemented in the \texttt{feval} command.

\subsection{Computation of Chebyshev--Fourier coefficients}\label{sec:diskcoeffs}
The low rank form of $\tilde{f}$ facilitates the use of fast transform methods based on the FFT. We can write the truncated tensor product 
Chebyshev--Fourier expansion of $ \tilde{f}$ as follows: 
\begin{align}
  \tilde{f}(\theta,\rho) \approx \sum_{k=-n/2}^{n/2-1}\sum_{\ell=0}^{m-1} X_{  \ell k }  T_{\ell}(\rho) e^{ik\theta}, \label{eq:2Dfourierchebyshevdisk}
\end{align}   
where $X$ is a matrix whose entries are the 2D Chebyshev--Fourier coefficients of $\tilde{f}$. Using the low rank form of $\tilde{f}$ 
given by~\eqref{eq: low_rank_rep}, 
the matrix $X$ can also be expressed in low rank form as  \smash{$X = ADB^T$}.
Here, $A$ is an $m \times K$ matrix whose $j$th column contains the coefficients $\{a_{\ell}^j\}$ from~\eqref{eq:1D_expansions}, 
$D$ is a $K$-by-$K$ diagonal matrix consisting of the pivot values $\{d_j\}$, and $B$ is an $n \times K$ matrix whose $j$th column 
contains the coefficients $\{b_{k}^j\}$ from~\eqref{eq:1D_expansions}. 
Given a sample of $\tilde{f}$ on an $m \times n$ Chebyshev--Fourier grid, the direct computation of the Chebyshev--Fourier coefficients 
of $\tilde{f}$ costs $\mathcal{O}(mn\log(mn))$ operations. However, using the GE procedure in Section~\ref{sec:structGE}, the low rank 
form of $X$ can be found in only $\mathcal{O}(K^3+ K^2(m+n) + K(m\log m + n \log n))$ operations. This is because once the GE process 
adaptively selects the skeleton representing $\tilde{f}$ at a cost of $\mathcal{O}(K^3+ K^2(m+n))$, the coefficients 
in~\eqref{eq:1D_expansions} for every $c_j(\rho)$ and $r_j(\theta)$ in~\eqref{eq: low_rank_rep} can be found in only 
$\mathcal{O}( K(m\log m + n \log n))$ operations. 

Several procedures, such as integration and differentiation, can be executed using the low rank factorization of $X$. Using the 
command \texttt{coeffs2} in Diskfun, $X$ can be explicitly computed with an additional $\mathcal{O}(Kmn)$ operations. 

The above operation retrieves coefficients when supplied with a sample of $\tilde{f}$, and the inverse of this operation provides an 
efficient way to sample $\tilde{f}$ on a $m \times n$ Chebyshev--Fourier grid. Given $X$ in low 
rank form, this proceeds in $\mathcal{O}(K(m\log m + n \log n))$ operations; the algorithm is implemented in the {\tt sample} command.  


\subsection{Integration}\label{sec:diskint}
To integrate $ \tilde{f}(\theta, \rho)$ over the unit disk, we again take advantage of the low rank form of~\eqref{eq: low_rank_rep}, 
transforming the double integral into sums of 1D integrals:  
\begin{equation}
 \int_{-\pi}^{\pi} \int_{0}^{1}   \tilde{f}(\theta, \rho) \rho \, d\rho  \, d\theta   \approx \sum_{j=1}^{K} d_j\int_{-\pi}^{\pi} r_j(\theta) \, d\theta \int_{0}^{1} c_j(\rho) \rho \, d\rho.
\end{equation}
For integration of the periodic $r_j(\theta)$ functions, the trapezoidal rule is used. To evaluate \smash{$\int_{0}^{1} c_j(\rho) \rho \, d\rho$}, 
the coefficients for $\rho c_j(\rho)$ are computed, and then Clenshaw-Curtis quadrature is applied~\cite[Ch. 19]{trefethen2013approximation}.
These $2K$ 1D integrals can be computed in a total of $\mathcal{O}(Km)$ operations. This can be further reduced using~\eqref{eq:EvenOddSplit} 
since only the even, $\pi$-periodic terms will contribute to the value of the integral.
 
Integration is implemented in the {\tt sum2} command. For example, the integral of $f(x,y) = -x^2-3xy -(y-1)^2$ over the unit disk 
is $-3\pi/2$, and can be computed in Diskfun as
\begin{verbatim}
f = diskfun(@(x,y) -x.^2-3*x.*y -(y-1).^2);
sum2(f)
ans =
   -4.712388980384692
\end{verbatim}
The error is determined with \texttt{abs(sum2(f)+3*pi/2)}, which gives $1.7764 \times 10^{-15}$. 

\subsection{Differentiation}\label{sec:diskdiff} 
When considering derivatives on the disk, note that partial differentiation with respect to $\rho$ can lead to artificial singularities 
at $\rho=0$. For example, if $f(\theta, \rho) = \rho^2$, then $\partial f/\partial \rho = 2\rho$ , which is not smooth on the 
disk. In contrast, for a smooth function $\tilde{f}$, partial derivatives with respect to $x$ and $y$ will always be well-defined. 
For this reason, and because of the usefulness of these operators in vector calculus (see Section~\ref{sec:vcalc}), 
we consider efficient and stable ways to calculate $\partial \tilde{f}/\partial x$ and $\partial \tilde{f}/\partial y$.   

By~\eqref{eq:transfvars}, $\rho = \sqrt{x^2 + y^2}$, and $ \theta = \tan^{-1}(y/x)$, so the chain rule can be applied to obtain 
\begin{align}
\label{eq:partialx}
\frac{\partial \tilde{f}}{\partial x} = \cos \theta \frac{\partial \tilde{f}}{\partial \rho} - \frac{1}{\rho } \sin \theta \frac{\partial \tilde{f}}{\partial \theta},\\
\label{eq:partialy}
\frac{\partial \tilde{f}}{\partial y} = \sin \theta \frac{\partial \tilde{f}}{\partial \rho} + \frac{1}{\rho } \cos \theta \frac{\partial \tilde{f}}{\partial \theta}.
\end{align} 
Exploiting the low rank form given in (\ref{eq: low_rank_rep}),~\eqref{eq:partialx} can be written as 
\begin{equation}
\frac{\partial \tilde{f}}{\partial x} \approx \sum_{j=1}^{K} d_j \bigg(\frac{\partial c_j(\rho )}{\partial \rho }\bigg) \bigg( \cos\theta \,r_j(\theta) \bigg)  - \sum_{j=1}^{K} d_j \bigg(\frac{ c_j(\rho )} {\rho }\bigg) \bigg(\sin\theta\, \frac{\partial r_j(\theta)} {\partial \theta}\bigg).
\label{eq:partial_diff_disk}
\end{equation}
A similar expression can be used for~\eqref{eq:partialy}.

Here we make an important observation. The above result establishes that approximants on the disk are continuously differentiable at $\rho =0$ only 
if \smash{$\sum_{j=1}^{K}c_j(\rho)$} is divisible by $\rho$. Suppose $\tilde{f}$ is nonzero at $\rho=0$ and write the approximant 
in the form given by~\eqref{eq:lowrankRepresentationSplit}. Then, because of~\eqref{eq: GE zero pole step}, for 
$2 \leq j \leq K^{+}$, each term $\deven_j \ceven_j(\rho)\reven_j(\theta)$ is zero at $\rho=0$. Since $\ceven_j(\rho)$ 
is an even Chebyshev polynomial, it must be of the form $\alpha_1\rho^2+\alpha_2\rho^4+\dots+\alpha_q\rho^{2q}$, 
where $q \leq \lfloor (m-1)/2\rfloor$. This implies that these functions are all divisible by $\rho$. For $j=1$, 
$\reven_1(\theta)$ is constant by~\eqref{eq: GE zero pole step}, and so all terms in~\eqref{eq:partial_diff_disk} 
involving derivatives of $\reven_1(\theta)$ with respect to $\theta$ vanish. Since every $\codd_j(\rho)$ function 
for $1\leq j\leq K^{-}$ is an odd function, these are also always divisible by $\rho$.  This means that the
approximants constructed by the BMC-II structure preserving GE procedure have inherited properties ensuring that they are continuously differentiable at $\rho=0$.  

There are $2K$ 1D derivatives to compute in~\eqref{eq:partial_diff_disk}. 
Using~\eqref{eq:1D_expansions}, 
\begin{align}
 &\sin\theta\,\frac{\partial r_j(\theta)}{\partial \theta} = \sum_{k=-n/2}^{n/2-1} \frac{-(k+1)b_{k+1}^j+(k-1)b_{k-1}^j}{2}  e^{i k \theta},\\
 &\cos\theta\,r_j(\theta) = \sum_{k=-n/2}^{n/2-1} \frac{b_{k+1}^j+b_{k-1}^j}{2}  e^{i k \theta},\end{align}
where $b_{-n/2-1}$ and $b_{n/2}$ are set to zero.
Expanding each $c_j(\rho)$ as in~\eqref{eq:1D_expansions}, the recursion formula in~\cite[p. 34]{mason2002chebyshev} gives
the coefficients for $\partial c_j(\rho )/\partial \rho $ in $\mathcal{O}(m)$ operations. 
To determine $c_j(\rho)/\rho$, we construct the operator $B_{\rho}$, which represents multiplication by the function $g(\rho) = \rho$ 
in the Chebyshev basis. Then, 
\begin{equation}
  \frac{c_j(\rho)}{\rho} = \sum_{\ell=0}^{m-1} (B_{\rho}^{-1}\underline{a}^j)_\ell T_\ell(\rho), \qquad  B_{\rho} = \begin{pmatrix}0 & \tfrac{1}{2}\cr 1 & 0 & \tfrac{1}{2} \cr & \tfrac{1}{2} & \ddots & \ddots & \cr & & \ddots & \ddots & \tfrac{1}{2}\cr  &&& \tfrac{1}{2} & 0 & \tfrac{1}{2} \cr &&&& \tfrac{1}{2} & 0\end{pmatrix},
\label{eq:Mr}
\end{equation}
where $\underline{a}^{j} = (a_{0}^{j},\ldots,a_{m-1}^{j})^T$. Here, $B_{\rho}^{-1}$ exists because we choose $B_\rho$ to be of size $m \times m$, 
where $m$ is an even integer. Working directly with the coefficients via~\eqref{eq:Mr} is an efficient way to bypass the artificial singularity 
introduced in~\eqref{eq:partial_diff_disk}, without explicitly avoiding computation at $\rho=0$. In contrast, the standard procedure when working on function values with the DFS method uses a "shifted grid" strategy~\cite{Fornberg_95_01,heinrichs2004spectral}.

Differentiation is accessed through the $\texttt{diff}$ command in Diskfun, and requires $\mathcal{O}(K(m+n))$ operations.     

\subsection{The $L_2$ norm and the weighted singular value decomposition}\label{sec:svd}
 In Diskfun, \texttt{norm(f)} is overloaded to compute the $L_2$ norm on the disk, which is the continuous analogue of the 
 matrix Frobenius norm~\cite{Townsend_15_01}. This is one of the very few instances in Diskfun where it makes more 
 sense to work with $f$ directly, rather than $\tilde{f}$. The $L_2$ norm of a function $f$ on the disk 
 is given in polar coordinates as
  \begin{equation}
  \|f\|_2^2 = \int_{-\pi}^{\pi}\int_{0}^{1}  |f(\theta, \rho)|^2 \rho \, d\rho \, d\theta.
  \label{eq:fronorm}
  \end{equation}
 Computing $\|f\|_2$ using~\eqref{eq:fronorm} directly is numerically unstable, especially when $f$ is near zero. A more stable formulation 
 is given in~\cite{Schmidt_1908_01}: If $f$ is $L_2$ integrable, then
   \begin{equation}
   \label{eq:svdfronorm}
      \|f\|_2^2 = \sum_{j=1}^{\infty}\sigma_j^2, 
   \end{equation}
 where $\sigma_1 \geq \sigma_2 \geq \cdots \geq 0 $ are real and nonnegative numbers referred to as the \textit{(weighted) singular values} 
 of $f$. For this reason, we are interested in the weighted SVD of $f$, which is given by
   \begin{equation} 
         \label{eq:SVD}
           f(\theta,\rho) = \sum_{j=1}^\infty \sigma_j u_j(\rho)v_j(\theta), \qquad (\theta,\rho)\in [-\pi, \pi]\times [0, 1].
   \end{equation} 
 The singular functions $\{u_j(\rho)\}$, $\rho \in [0, 1]$, and $\{v_j(\theta)\}$, $\theta \in [-\pi, \pi]$, are orthonormal
 under the following inner products, respectively: 
  \begin{equation}
          \label{eq:diskinnerprod}
          <u, s>_\rho = \int_0^{1}u(\rho)\overline{s(\rho)} \rho  \, d\rho, \qquad
          <v, w> = \int_{-\pi}^{\pi} v(\theta)\overline{w(\theta)} \, d \theta, 
   \end{equation}
where the bars on $s$ and $w$ denote complex conjugation. 

The weighted SVD for a function on the disk is determined by applying a generalization of $QR$ factorization to quasimatrices.
Restricting the low rank approximation to $\tilde{f}$ given by~\eqref{eq: low_rank_rep} to $(\theta, \rho) \in [-\pi, \pi] \times [0, 1]$, we form a $[0, 1] \times K$ quasimatrix $C$ such that the $j$th column of $C$ is $c_j(\rho)$ in~\eqref{eq: low_rank_rep} restricted to the domain $[0, 1]$. Similarly, we form the $[-\pi, \pi] \times K$ quasimatrix $R$ such that the $j$th column of $R$ is $r_j(\theta)$. 
A $QR$ quasimatrix factorization with respect to the standard $L_2$ inner product on $[-\pi, \pi] \times [0, 1]$ is given in~\cite{trefethen2009householder} and selects the Legendre polynomials to orthogonalize against, and this procedure is applied to $R$. In consideration of~\eqref{eq:diskinnerprod}, $C$ is orthogonalized against the functions 
\[ 
\frac{\sqrt{2}}{J_1(\omega_{k})}J_{0}(\omega_{k}\rho), \qquad k = 1, 2, \dots,
\] 
where $J_{\nu}$ is the Bessel function of order $\nu$, and $\omega_k$ is the $k$th positive root of $J_0(\rho)$. 
This finds $\{u_j(\rho)\}$, which are orthonormal with respect to~\eqref{eq:diskinnerprod}. Once the $QR$ factorizations 
for $C$ and $R$ are known, the SVD is determined through standard techniques, as discussed in~\cite{Townsend_15_01}.  
  
In addition to providing a mathematically stable way to compute~\eqref{eq:fronorm}, the weighted SVD gives the 
best rank $K$ approximation to $f$ with respect to the $L_2$ inner product on the disk. Unfortunately, the 
use of the weighted SVD as a low-rank approximation method is limited because the rank $1$ terms in~\eqref{eq:SVD} may 
be discontinuous at the origin of the disk~\cite{wilber2016}, and consequently, the truncation of~\eqref{eq:SVD} may not be smooth. 
The SVD is accessed in Diskfun through the {\tt svd} command, and is used internally in the {\tt norm} command.
  
\subsection{Vector-valued functions and vector calculus on the disk} \label{sec:vcalc}
Vector-valued functions can also be constructed in Diskfun. These functions are represented with respect to the Cartesian coordinate 
basis vectors $\hat{\mathbf{i}}$ and $\hat{\mathbf{j}}$, since not all smooth vector fields defined over the disk have smooth components 
when represented with respect to the polar coordinate basis vectors, $\hat{\mathbf{r}}$ and $\hat{\boldsymbol{\theta}}$. For example, 
the vector field given by $\mathbf{f} = 0\hat{\mathbf{i}}+\hat{\mathbf{j}}$ is expressed as 
$\mathbf{f} = \sin\theta \hat{\mathbf{r}}+ \cos \theta \hat{\boldsymbol{\theta}}$ in polar coordinates, and both of these components are
discontinuous at the origin of the disk. 
 
Vector-valued functions are accessed in Diskfun through the creation of diskfunv objects. A diskfunv consists of two diskfun objects, 
one for each component of the vector-valued function. Algorithms involving diskfunv objects are implemented for algebraic actions, 
such as addition, as well as vector-based operations, such as the dot/cross products, and  divergence. Commands that map scalar-valued functions to vector-valued functions and vice-versa, such as { \tt grad(f)} 
and { \tt curl(f)}, are also included. In the latter case, the standard interpretations are used, i.e., $\nabla \times f = [ f_y, -f_x]$ 
for a scalar function $f$, and $\nabla \times \mathbf{u} = v_x-u_y$ when $\mathbf{u}=[u,v]$ is a vector-valued function. 
\begin{figure} 
\begin{center}
 \begin{minipage}{.4\textwidth}
  \includegraphics[width=0.8\textwidth]{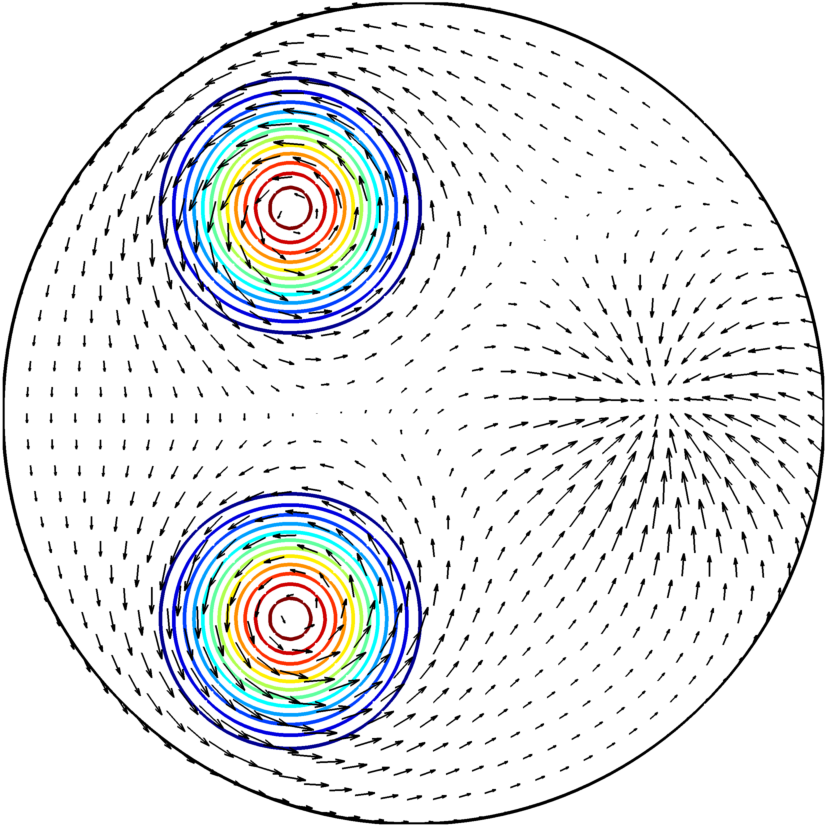}
 \end{minipage}
 \begin{minipage}{.4\textwidth}
  \includegraphics[width=0.8\textwidth]{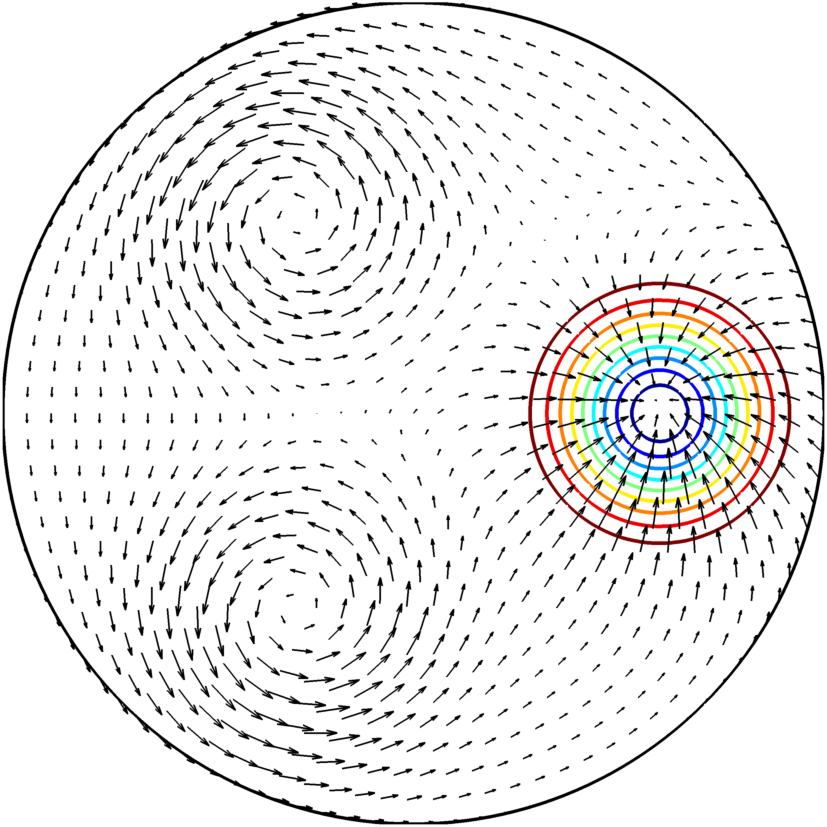}
 \end{minipage}
 \end{center}
 \caption{
The vector function $\mathbf{u} = \nabla \times \psi + \nabla \phi$, with $\psi$ and $\phi$ defined in~\eqref{eq:potentials}, together with its curl, $\nabla \times \mathbf{u}$ (left), and divergence, $\nabla \cdot \mathbf{u}$ (right). The field was plotted using {\tt quiver(u)}, while the curl and divergence were computed using {\tt curl(u)} and {\tt div(u)}, respectively, and plotted using the {\tt contour} command.}
\label{fig:vort_div} 
\end{figure}
As an example, consider the potential functions given by
\begin{equation} 
\begin{split} &\psi(x,y) = 10e^{-10(x+.3)^2-10(y+.5)^2}+10e^{-10(x+.3)^2-10(y-.5)^2} +15(1-x^2-y^2),\\  
&\phi(x,y) =  10e^{-10(x-.6)^2-40y^2},
\end{split} \label{eq:potentials} \end{equation}
and the vector field $\mathbf{u} = \nabla \times \psi + \nabla\phi$. This field
consists of the sum of a divergence-free term, $\nabla \times \psi$, and a curl-free term, $\nabla\phi$. Once $\psi$ and 
$\phi$ are constructed as diskfun objects, $\mathbf{u}$ can be constructed with a single line of code:{ \tt u = curl(psi)+grad(phi)}. 
Figure~\ref{fig:vort_div} displays a plot of $\mathbf{u}$ together with its curl and divergence.

\subsection{Miscellaneous operations}
Diskfun is included as an object class in Chebfun, and so has access to many of the operations in Chebfun. Operations that do 
not strictly require symmetry properties related to the geometry of the disk are computed using Chebfun2 with functions defined 
in polar coordinates~\cite{Chebfun2}. This includes optimization routines, such 
as {\tt min2}, {\tt max2}, and {\tt roots}, as well as procedures inspired by matrices such as {\tt trace} and {\tt lu}. Operations 
that use Chebfun2 are performed automatically, without requiring adjustments or intervention by the user. 
 
\section{A fast Poisson solver for computing solutions in low rank form}\label{sec:Poisson} 
In~\cite{wilber2016} and~\cite{shen2000new}, optimal complexity solvers for Poisson's equation on the disk are formulated through the use of parity properties associated with the Chebyshev--Fourier coefficients of BMC-II functions. 
Unfortunately, these solvers cannot capitalize on the low rank structure of the approximants in~\eqref{eq: low_rank_rep}, and they do not guarantee that the computed solution has good compression properties. Finding a low rank representation of the solution requires additional work, and such representations are essential in Diskfun. This has motivated the development of a fast Poisson solver that directly computes low rank approximations to solutions. 

Our method uses the factored alternating direction implicit (ADI) method~\cite{benner2009adi, li2002low} to work independently on the Chebyshev and Fourier coefficients in~\eqref{eq:1D_expansions}. We combine ADI with the Fourier and ultraspherical spectral methods, so that every linear system we solve is sparse and  spectral accuracy is guaranteed~\cite{olver2013fast}. We find that the ADI-based method efficiently constructs low rank solutions  whenever the numerical rank of the forcing function is sufficiently low.


Given a function $f(\theta, \rho)$ on the unit disk, we seek the solution $u(\theta, \rho)$ to Poisson's equation, 
$\nabla^2 u =f$, where $(\theta, \rho) \in [-\pi, \pi]\times [0, 1]$. To ensure a unique solution, Dirichlet conditions 
are prescribed as $u(\theta, 1) = g(\theta)$, where $g$ is a $2\pi$-periodic function.  In this section, we will assume that $g(\theta) = 0$.\footnote{Whenever $g(\theta)$ is nonzero, the system can be solved by relating it to a system with homogeneous boundary conditions (see~\cite[Ch. 6]{boyd2001chebyshev}).}

To enforce that the numerical solution $u$ is continuous over $u(\theta, 0)$, we apply the disk analogue to the DFS method and 
consider solving the related equation $\nabla^2 \tilde{u} =\tilde{f}$, where $\tilde{f}$ is the BMC-II extension of $f$ 
given by~\eqref{eq:BMCsym disk}.  The equation $\nabla^2 \tilde{u} =\tilde{f}$ is expressed in polar coordinates as   
\begin{equation}
 \rho^2 \frac{\partial^2 \tilde{u}}{\partial \rho^2} + \rho \frac{\partial \tilde{u}}{\partial \rho} + \frac{\partial^2 \tilde{u}}{\partial \theta^2} = \rho^2 \tilde{f}, \qquad (\theta, \rho)\in[-\pi,\pi]\times[-1, 1],
\label{eq:PoissonDiskBMC}
\end{equation} 
 where the standard formulation is multiplied by $\rho^2$ so that the variable 
 coefficients are low degree polynomials in $\rho$. It is straightforward to show that $\tilde{u}$ must also possess 
 BMC-II symmetry and therefore corresponds to a differentiable function on the disk. Restricting $\tilde{u}$ to $[-\pi, \pi]\times[0, 1]$ gives $u$.

To ensure that $\tilde{u}$ satisfies homogeneous boundary conditions, we will express it as a product of $1-\rho^2$ and an unknown function $\hat{u}$ . Expanding $\hat{u}$ in the Chebyshev--Fourier basis, we find that
\begin{equation}
\label{eq: FourierDirichlet}
  \tilde{u} (\theta,\rho) =  (1-\rho^2) \hat{u}(\theta, \rho) \approx (1-\rho^2)   \sum_{k=-n/2}^{n/2-1} \sum_{\ell=0}^{m-1} Y_{  \ell k }  T_{\ell}(\rho) e^{ik\theta},
\end{equation}
where $n$ is an even integer. 

We seek a low rank approximation to the Chebyshev--Fourier coefficient matrix $Y \in \mathbb{C}^{m \times n}$. Since~$1-\rho^2 =  (T_0(\rho) -  T_2(\rho))/2$, we can represent multiplication by $1-\rho^2$ in the Chebyshev basis with a sparse operator $M$. Then, $MY$ is the Chebyshev--Fourier coefficient matrix of $\tilde{u}$, i.e., $MY = X$ in~\eqref{eq:2Dfourierchebyshevdisk}.

To use ADI, the discretization of~\eqref{eq:PoissonDiskBMC} must be expressed as a Sylvester matrix equation of the form $AY-YB=C$, with the matrices $A \in \mathbb{C}^{m \times m}$ and $B \in \mathbb{C}^{n \times n}$ represented in a data-sparse way.  
 Plugging~\eqref{eq: FourierDirichlet} into~\eqref{eq:PoissonDiskBMC} and applying the chain rule, we rewrite~\eqref{eq:PoissonDiskBMC} with respect to $\hat{u}$: 
\begin{equation}
\label{eq:Lmop}
 \underbrace{\rho^2(1-\rho^2)\dfrac{\partial^2\hat{u} }{\partial \rho^2}+ (-5\rho^3+\rho)\dfrac{\partial \hat{u} }{\partial \rho}  -4\rho^2 \hat{u}}_{=\mathcal{L}}+ (1-\rho^2)\frac{\partial^2 \hat{u}}{\partial \theta^2} = \rho^2\tilde{f}.
\end{equation}
We now seek a discrete counterpart to the operator $\mathcal{L}$ that acts on the Chebyshev coefficients of $\hat{u}$.  To formulate such an operator, we apply a variant of the ultraspherical spectral method~\cite{olver2013fast}. This method uses recurrence relations between the Chebyshev and other ultraspherical polynomials to define sparse differential operators. Applying the ultraspherical spectral method directly results in a discretization of $\mathcal{L}$ that is sparse and banded. 
However, the bandwidth of this operator can be further reduced if we use a recurrence relation between the Chebyshev polynomials of the first and second kind that involves the term $1-\rho^2$.  Using~\cite[(18.9.10)]{NIST}, we have that 
\begin{equation}
\label{eq:ident}
(1-\rho^2)\dfrac{d^2}{d \rho^2} T_\ell(\rho) = -\ell(\ell+1)\left( \frac12U_\ell(\rho) - \frac12U_{\ell-2}(\rho)\right) +\ell U_\ell(\rho),  \qquad \ell \geq 2, 
\end{equation}
where $\{ U_\ell\}$ are the Chebyshev polynomials of the first kind. We use~\eqref{eq:ident} to define a discrete operator $D_{2(1)}$ that represents $(1-\rho^2) \partial^2/\partial \rho^2$.  Like all differentiation operators in the the ultraspherical spectral method, $D_{2(1)}$ acts on coefficients in one basis and converts them to another. Specifically, it acts on Chebyshev coefficients and returns coefficients in the $\{U_\ell\}$ basis. 
The remaining terms in $\mathcal{L}$ are expressed using standard techniques in the ultraspherical spectral method, and the resulting discretization of $\mathcal{L}$, denoted as $L$, is a banded matrix of bandwidth 4.  

We will use $L$ and the differentiation matrix 
\[D_{F}^2 = \textnormal{diag}\left(-(\tfrac{n}{2})^2, -(\tfrac{n-1}{2})^2, \cdots, 0, -1, -4, \cdots , -(\tfrac{n-1}{2})^2\right),\] which discretizes $\partial^2 /\partial \theta^2$ and acts on Fourier coefficients, to write the discretization of~\eqref{eq:Lmop} as a generalized Sylvester equation:
\begin{equation} \label{eq:sylvesterPoisson}
  LY+ S_1MYD_{F}^2 = S_1M_{\textnormal{\tiny{$\rho^2$}}} F .
\end{equation}
Recall that the matrix $M$ is an operator representing multiplication by $1-\rho^2$ .\footnote{Note that in the first term, multiplication by $1-\rho^2$ occurs implicitly via~\eqref{eq:Lmop}.} The tridiagonal matrix $S_1$ converts coefficients in the Chebyshev basis to the $\{ U_{\ell} \}$ basis; this is required due to the action of $L$ (see~\cite{olver2013fast}).  On the right-hand side, $F$ is the Chebyshev--Fourier matrix of coefficients for \smash{$\tilde{f}$} , and \smash{$M_{ \textnormal{\tiny{$\rho^2$}}} $} is a tridiagonal matrix representing multiplication by $\rho^2$. 

To apply ADI, we must write~\eqref{eq:sylvesterPoisson} in the following form:
\begin{equation}
\label{eq:dispPoisson}
 \underbrace{(S_1M)^{-1} L}_{=A}Y- Y\underbrace{(- D_{F}^2)}_{=B} = \underbrace{M^{-1}M_{\textnormal{\tiny{$\rho^2$}}}  F}_{=C}
 \end{equation}
The matrices $L$ and  $S_1M$ are each banded with a bandwidth of 4, and  $B$ is diagonal. 
We solve~\eqref{eq:sylvesterPoisson} by applying the factored ADI method in~\cite{benner2009adi}.  This method never requires $F$ to be formed explicitly. Rather, it operates directly on the low rank factorization of $F$ described in Section~\ref{sec:diskcoeffs}. The solution is returned as a low rank factorization,  $Y = ZDG^*$, where $Z$ is a collection of Chebyshev coefficients, $D$ is diagonal, and $G$ is a collection of Fourier coefficients. 

ADI is an iterative method, and the convergence of the method is sensitive to the selection of a set of shift parameters~\cite{ lu1991solution, sabino2006solution}. The spectrum of $A$ in~\eqref{eq:dispPoisson}, denoted as $\sigma(A)$, can be contained in an interval on the real line that is well-separated from the interval containing $\sigma(B)$.  
In such a scenario, near-optimal shift parameters are known and can efficiently be computed~\cite{sabino2006solution}.

The computational cost of ADI is dependent on the  rank of the matrix $C$ and properties of the matrices $A$ and $B$.  
If $A$ and $B$ were normal, one could directly apply bounds given in~\cite{lu1991solution, beckermann2016singular, sabino2006solution} to find the maximum number of ADI iterations required for approximating $Y$ to within the tolerance $\varepsilon$.\footnote{ Bounds are also supplied in~\cite{beckermann2016singular} and~\cite{sabino2006solution} for the case of non-normal $A$ and $B$ through the use of pseudospectra and fields of values, respectively. In our case, the matrix $V$ in the eigendecomposition $A = V\Lambda V^{-1}$ is well-conditioned, and we therefore only require a slight generalization on the bounds  supplied for normal operators. } 
However, $A$ is not a normal matrix. Fortunately, the matrix $V$ in the eigendecomposition $A = V\Lambda V^{-1}$ is well-conditioned, with $\kappa_2(V) = \|V\|_2\|V^{-1}\|_2$  growing approximately quadratically with $m$. We apply the bound for normal matrices given in~\cite{beckermann2016singular} to the eigendecomposition of $A$ and find that we require at most $N$ steps of ADI, where $N =  \left\lceil \pi^{-2}\log(4 \kappa_2(V)  / \varepsilon)\log(16\gamma)\right\rceil $. Here, $\gamma$, described in Corollary 4.2 of~\cite{beckermann2016singular}, is a function of $\sigma(A)$ and $\sigma(B)$.  Empirically, we observe that $\gamma$ grows slightly faster than quadratically as $(m+n)$ increases.  If $C$ is of rank $K$, then each iteration of ADI requires $2K$ sparse, linear solves, so the total cost for performing factorized ADI on~\eqref{eq:dispPoisson} is $\mathcal{O}(NK(m+n))$.

The ADI method results in an overestimation of the numerical rank of $Y$. This is remedied by applying a compression step on the factorization $Y = ZDG^*$ via the SVD, at a computational cost of \smash{$\mathcal{O}( (NK)^2(m+n) +(NK)^3)$}. Accounting for the logarithmic growth of $N$, the overall cost of our procedure is $\mathcal{O}(K^2 (n+m) (\log(n)\log(m))^2  + K^3(\log(n)\log(m))^3)$.

In contrast, optimal complexity methods that ignore the numerical rank of $\tilde{f}$ find  a low rank approximation to $Y$ in $\mathcal{O}(mn \log mn + \tilde{K}^3 + \tilde{K}^2(m+n))$, where $\tilde{K}$ is the numerical rank of $Y$. This is because one can decouple~\eqref{eq:dispPoisson} and find the coefficient matrix $Y$ in $\mathcal{O}(mn)$ operations. 
A low rank approximation to $Y$ can then be constructed by retrieving  the function values associated with $Y$ via the FFT, and then performing BMC structure-preserving GE.

The ADI-based method is beneficial  when the numerical rank of $Y$ is sufficiently small, and in practice, we use the alternative solver described in~\cite{wilber2016} whenever ADI is not advantageous. Figure~\ref{fig:FastPoissonSolver} (left) compares the rate at which these two methods construct a low rank approximation, represented as a diskfun object, to the solution of~\eqref{eq:PoissonDiskBMC}. For choices of $\tilde{f}$ with various numerical ranks, we plot the wall clock time in seconds against increasingly large values of $n$, with $m=2n+1$. The alternative solver, which is insensitive to the rank of $\tilde{f}$, is represented in black. The ADI-based method proves effective for moderate-sized problems ($n=1048$) when the rank of $\tilde{f}$ is below $10$, performing up to $5$ times faster than the alternative method. With $n = 10,\! 000$ and $\tilde{f}$ of numerical rank $5$,  the ADI solver constructs a low rank solution in under 5 seconds.\footnote{Timings were performed in MATLAB R2016a on a 2015 Macbook Pro with no explicit parallelization. The degrees of freedom used in this experiment were increased artificially to demonstrate asymptotic complexity. } 
 \begin{figure} 
  \begin{minipage}{.55\textwidth} 
  \begin{overpic}[width=\textwidth]{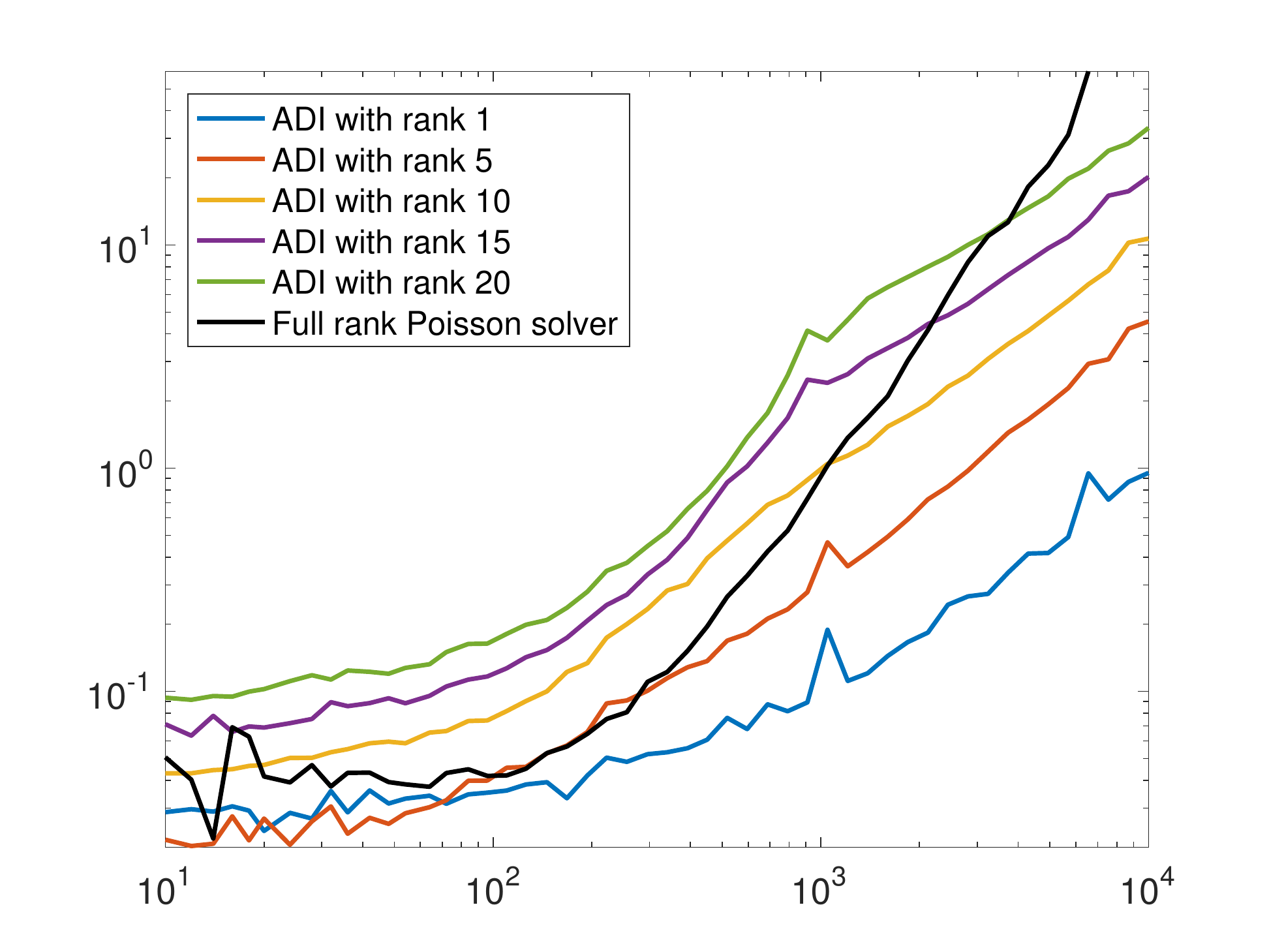}
  \put(1,18){\rotatebox{90}{\small{elapsed time (secs)}}}
  \put(50,0){$n$}
  \end{overpic}
  \end{minipage}
  \begin{minipage}{.44\textwidth} 
  \begin{overpic}[width=\textwidth]{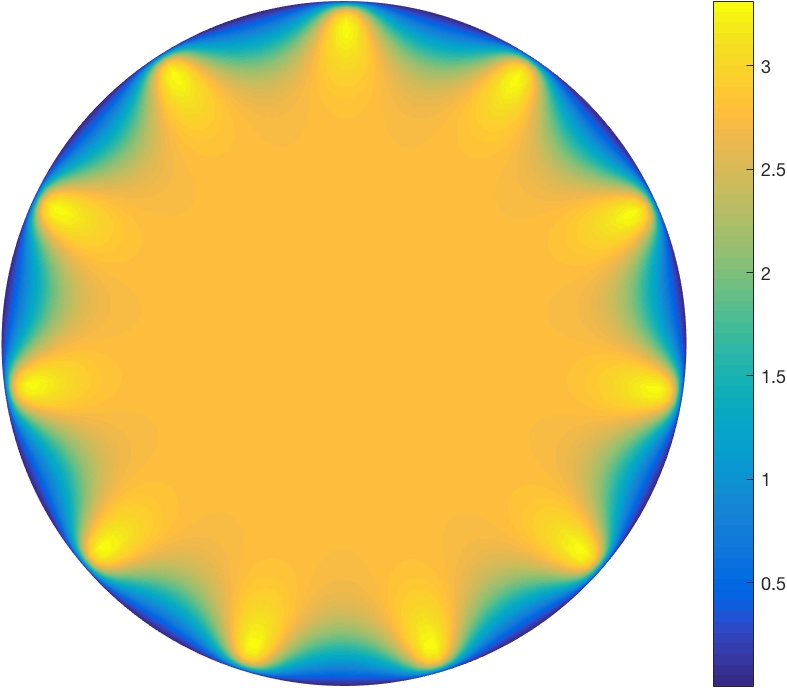}
  \end{overpic}
  \end{minipage}
  \caption{Left: Comparison of the execution (wall clock) time for the ADI-based Poisson solver and an optimal complexity solver that does not account for low rank structures (black), as a function of $n$, where the  problem size is $(2n +1)\times n$.  Timings include the construction of a diskfun object. 
  Right: Solution to $\nabla ^2 u = f$ with boundary condition $u(\theta, 1) = 0$, where $f$ is given in~\eqref{eq:PoissonRHS}. }
  \label{fig:FastPoissonSolver}
 \end{figure}
 
Our solver is implemented in Diskfun in an integrated way: The output returned is automatically represented as a diskfun object, and can therefore immediately be visualized or operated on using Diskfun commands. For example, Figure~\ref{fig:FastPoissonSolver} (right) displays the solution to $\nabla^2 u = f$ computed with the {\tt poisson} command in Diskfun. 
Here, $f$ is numerically a rank $16$ function, given by
\begin{equation}
f(\theta, \rho) = e^{-40(\rho^2-1)^4}\sinh\big(5-5\rho^{11}\cos(11\theta-11/\sqrt{2})\big), 
\label{eq:PoissonRHS}
\end{equation} 
and the boundary condition is $u(\theta, 1) = 0$. 
  
\section{Conclusions}\label{sec:concl}
The analogue of the double Fourier sphere (DFS) method for functions on the unit disk provides a useful structure that is 
retained through a new iterative Gaussian elimination procedure on functions. We use this concept to construct low 
rank approximations to functions on the disk that facilitate fast and stable 
computations based on the FFT. Fast and spectrally accurate algorithms exploiting low rank structures are described for several operations, including differentiation, integration, vector calculus, and the solving of Poisson's equation. We have
implemented these ideas in Diskfun, which is part of the publicly available, open-source software  Chebfun. This allows 
investigators to compute with functions in polar geometries in an intuitive, accurate, and highly efficient way, without concern for the underlying  discretization procedure.  

\section*{Acknowledgments}
We are grateful to Nick Trefethen for his detailed comments on a draft of the paper. We thank Nick Hale and Stefan G\"{u}ttel for observations 
concerning the computation of the weighted SVD in Section~\ref{sec:svd}, and Jared Aurentz for valuable feedback. We thank Behnam Hashemi and the Chebfun team for reviewing the Diskfun software. We thank the editor and referees for their valuable comments, and are particularly appreciative of an anonymous reviewer of~\cite{townsend2015computing}, whose comments motivated the development of our fast disk Poisson solver. 

\bibliographystyle{siam}
\bibliography{diskfun}

\begin{thebibliography}{10}

\bibitem{amore2008solving}
{\sc P.~Amore}, {\em Solving the {H}elmholtz equation for membranes of
  arbitrary shape: numerical results}, J. of Phys. A: Mathematical and
  Theoretical, 41 (2008), pp.~265--206.

\bibitem{bebendorf2000approximation}
{\sc M.~Bebendorf}, {\em Approximation of boundary element matrices}, Numer.
  Math., 86 (2000), pp.~565--589.

\bibitem{beckermann2016singular}
{\sc B.~Beckermann and A.~Townsend}, {\em On the singular values of matrices
  with displacement structure}, SIAM J. Matrix Anal. Appl.,  (2016).
\newblock To appear.

\bibitem{benner2009adi}
{\sc P.~Benner, R.-C. Li, and N.~Truhar}, {\em On the {ADI} method for
  {S}ylvester equations}, J. Comput. Appl. Math., 233 (2009), pp.~1035--1045.

\bibitem{bhatia1954circle}
{\sc A.~Bhatia and E.~Wolf}, {\em On the circle polynomials of {Z}ernike and
  related orthogonal sets}, in Mathematical Proceedings of the Cambridge
  Philosophical Society, vol.~50, Cambridge Univ Press, 1954, pp.~40--48.

\bibitem{boyd2001chebyshev}
{\sc J.~P. Boyd}, {\em Chebyshev and Fourier Spectral Methods}, Courier
  Corporation, 2001.

\bibitem{boyd2011comparing}
{\sc J.~P. Boyd and F.~Yu}, {\em Comparing seven spectral methods for
  interpolation and for solving the {P}oisson equation in a disk: {Z}ernike
  polynomials, {L}ogan--{S}hepp ridge polynomials, {C}hebyshev--{F}ourier
  series, cylindrical {R}obert functions, {B}essel--{F}ourier expansions,
  square-to-disk conformal mapping and radial basis functions}, J. Comput.
  Phys., 230 (2011), pp.~1408--1438.

\bibitem{carvajal2005hybrid}
{\sc O.~A. Carvajal, F.~W. Chapman, and K.~O. Geddes}, {\em Hybrid
  symbolic-numeric integration in multiple dimensions via tensor-product
  series}, in Proceedings of the 2005 international symposium on symbolic and
  algebraic computation, ACM, 2005, pp.~84--91.

\bibitem{churchill1961fourier}
{\sc R.~Churchill}, {\em Fourier Series and Boundary Value Problems},
  McGraw-Hill book Company, Incorporated, 1941.

\bibitem{Chebfun}
{\sc T.~A. Driscoll, N.~Hale, and L.~N. Trefethen}, eds., {\em Chebfun Guide},
  Pafnuty Publications, Oxford, 2014.

\bibitem{eisen1991spectral}
{\sc H.~Eisen, W.~Heinrichs, and K.~Witsch}, {\em Spectral collocation methods
  and polar coordinate singularities}, J. Comput. Phys., 96 (1991),
  pp.~241--257.

\bibitem{Fornberg_95_01}
{\sc B.~Fornberg}, {\em A pseudospectral approach for polar and spherical
  geometries}, SIAM J. Sci. Comp., 16 (1995), pp.~1071--1081.

\bibitem{FFBook}
{\sc B.~Fornberg and N.~Flyer}, {\em A {P}rimer on {R}adial {B}asis {F}unctions
  with {A}pplications to the {G}eosciences}, {SIAM}, Philadelphia, 2015.

\bibitem{foster2006comparison}
{\sc L.~V. Foster and X.~Liu}, {\em Comparison of rank revealing algorithms
  applied to matrices with well defined numerical ranks}, 2006.

\bibitem{godon1997numerical}
{\sc P.~Godon}, {\em Numerical modeling of tidal effects in polytropic
  accretion disks}, The Astrophysical Journal, 480 (1997), p.~329.

\bibitem{Golub_2012_01}
{\sc G.~H. Golub and C.~F. Van~Loan}, {\em Matrix Computations}, Johns Hopkins
  University Press, 2012.
\newblock 4th edition.

\bibitem{goreinov1997theory}
{\sc S.~A. Goreinov, E.~E. Tyrtyshnikov, and N.~L. Zamarashkin}, {\em A theory
  of pseudoskeleton approximations}, Linear Algebra and its Applications, 261
  (1997), pp.~1--21.

\bibitem{halko2011finding}
{\sc N.~Halko, P.-G. Martinsson, and J.~A. Tropp}, {\em Finding structure with
  randomness: Probabilistic algorithms for constructing approximate matrix
  decompositions}, {SIAM} review, 53 (2011), pp.~217--288.

\bibitem{heinrichs2004spectral}
{\sc W.~Heinrichs}, {\em Spectral collocation schemes on the unit disc}, J.
  Comput. Phys., 199 (2004), pp.~66--86.

\bibitem{Heryudono2010}
{\sc A.~R.~H. Heryudono and T.~A. Driscoll}, {\em Radial basis function
  interpolation on irregular domains through conformal transplantation}, J.
  Sci. Comput., 44 (2010), pp.~286--300.

\bibitem{kapurl1995algorithm}
{\sc S.~Kapurl}, {\em An algorithm for the fast {H}ankel transform}, 1995.
\newblock Yale technical report.

\bibitem{Karageorghis2007304}
{\sc A.~Karageorghis, C.~Chen, and Y.-S. Smyrlis}, {\em A matrix decomposition
  {RBF} algorithm: Approximation of functions and their derivatives}, Appl.
  Numer. Math., 57 (2007), pp.~304--319.

\bibitem{kerswell2005recent}
{\sc R.~Kerswell}, {\em Recent progress in understanding the transition to
  turbulence in a pipe}, Nonlinearity, 18 (2005), p.~R17.

\bibitem{li2002low}
{\sc J.-R. Li and J.~White}, {\em Low rank solution of {L}yapunov equations},
  SIAM J. Matrix Anal. Appl., 24 (2002), pp.~260--280.

\bibitem{lu1991solution}
{\sc A.~Lu and E.~L. Wachspress}, {\em Solution of {L}yapunov equations by
  alternating direction implicit iteration}, Comp. \& Math. with Appl., 21
  (1991), pp.~43--58.

\bibitem{mahajan2007orthonormal}
{\sc V.~N. Mahajan and G.~Dai}, {\em Orthonormal polynomials in wavefront
  analysis: analytical solution}, JOSA A, 24 (2007), pp.~2994--3016.

\bibitem{martin2012transformation}
{\sc G.~Martin}, {\em Transformation Geometry: An Introduction to Symmetry},
  Springer, New York, 2012.

\bibitem{mason2002chebyshev}
{\sc J.~C. Mason and D.~C. Handscomb}, {\em Chebyshev Polynomials}, CRC Press,
  2002.

\bibitem{Merilees_73_01}
{\sc P.~E. Merilees}, {\em The pseudospectral approximation applied to the
  shallow water equations on a sphere}, Atmosphere, 11 (1973), pp.~13--20.

\bibitem{NIST}
{\sc F.~W. Olver, D.~W. Lozier, R.~F. Boisver, and C.~W. Clark}, {\em NIST
  Handbook of Mathematical Functions}, Cambridge University Press, 2010.

\bibitem{olver2013fast}
{\sc S.~Olver and A.~Townsend}, {\em A fast and well-conditioned spectral
  method}, {SIAM} Review, 55 (2013), pp.~462--489.

\bibitem{o2010algorithm}
{\sc M.~O'Neil, F.~Woolfe, and V.~Rokhlin}, {\em An algorithm for the rapid
  evaluation of special function transforms}, App. Comp. Harm. Analy., 28
  (2010), pp.~203--226.

\bibitem{pringle1981accretion}
{\sc J.~Pringle}, {\em Accretion discs in astrophysics}, Annual Review of
  Astronomy and Astrophysics, 19 (1981), pp.~137--162.

\bibitem{sabino2006solution}
{\sc J.~Sabino}, {\em Solution of large-scale {L}yapunov equations via the
  block modified {S}mith method}, PhD thesis, Rice University, 2006.

\bibitem{Schmidt_1908_01}
{\sc E.~Schmidt}, {\em Zur {T}heorie der linearen und nichtlinearen
  {I}ntegralgleichungen. iii. {T}eil}, Mathematische Annalen, 65 (1908),
  pp.~370--399.

\bibitem{schwarz1869ueber}
{\sc H.~A. Schwarz}, {\em Ueber einige {A}bbildungsaufgaben}, Journal f{\"u}r
  die reine und angewandte Mathematik, 70 (1869), pp.~105--120.

\bibitem{serre2001three}
{\sc E.~Serre and J.~Pulicani}, {\em A three-dimensional pseudospectral method
  for rotating flows in a cylinder}, Computers and Fluids, 30 (2001),
  pp.~491--519.

\bibitem{shen2000new}
{\sc J.~Shen}, {\em A new fast {C}hebyshev--{F}ourier algorithm for
  {P}oisson-type equations in polar geometries}, Appl. Numer. Math., 33 (2000),
  pp.~183--190.

\bibitem{townsend2014computing}
{\sc A.~Townsend}, {\em Computing with functions in two dimensions}, PhD
  thesis, University of Oxford, 2014.

\bibitem{townsend2015fast}
\leavevmode\vrule height 2pt depth -1.6pt width 23pt, {\em A fast
  analysis-based discrete {H}ankel transform using asymptotic expansions}, SIAM
  J. Numer. Anal., 53 (2015), pp.~1897--1917.

\bibitem{townsend2016gaussian}
\leavevmode\vrule height 2pt depth -1.6pt width 23pt, {\em Gaussian elimination
  corrects pivoting mistakes}, arXiv preprint arXiv:1602.06602,  (2016).

\bibitem{Chebfun2}
{\sc A.~Townsend and L.~N. Trefethen}, {\em An extension of {C}hebfun to two
  dimensions}, SIAM J. Sci. Comp., 35 (2013), pp.~C495--C518.

\bibitem{Townsend_15_01}
\leavevmode\vrule height 2pt depth -1.6pt width 23pt, {\em Continuous analogues
  of matrix factorizations}, in Proc. Royal Soc. A, vol.~471, 2015, pp.~1--21.

\bibitem{townsend2015computing}
{\sc A.~Townsend, H.~Wilber, and G.~B. Wright}, {\em Computing with functions
  in spherical and polar geometries {I}. {T}he sphere}, SIAM J. Sci. Comp.,
  38-4 (2016), pp.~C403--C425.

\bibitem{trefethen2000spectral}
{\sc L.~N. Trefethen}, {\em Spectral Methods in MATLAB}, {SIAM}, 2000.

\bibitem{trefethen2009householder}
\leavevmode\vrule height 2pt depth -1.6pt width 23pt, {\em Householder
  triangularization of a quasimatrix}, IMA J. Numer. Anal.,  (2009), p.~drp018.

\bibitem{trefethen2013approximation}
\leavevmode\vrule height 2pt depth -1.6pt width 23pt, {\em Approximation Theory
  and Approximation Practice}, {SIAM}, 2013.

\bibitem{vasil2016tensor}
{\sc G.~M. Vasil, K.~J. Burns, D.~Lecoanet, S.~Olver, B.~P. Brown, and J.~S.
  Oishi}, {\em Tensor calculus in polar coordinates using jacobi polynomials},
  J. Comput. Phys., 325 (2016), pp.~53--73.

\bibitem{wilber2016}
{\sc H.~Wilber}, {\em Numerical computing with functions on the sphere and
  disk}, Master's thesis, Boise State University, 2016.

\bibitem{wright2015extension}
{\sc G.~B. Wright, M.~Javed, H.~Montanelli, and L.~N. Trefethen}, {\em
  Extension of {C}hebfun to periodic functions}, SIAM J. Sci. Comp., 37 (2015),
  pp.~C554--C573.

\bibitem{von1934beugungstheorie}
{\sc F.~Zernike}, {\em Beugungstheorie des schneidenver fahrens und seiner
  verbesserten form, der phasenkontrastmethode}, Physica, 1 (1934),
  pp.~689--704.

\end{thebibliography}

\end{document}